\documentclass[12pt]{amsart}

\usepackage{amssymb}
\usepackage{latexsym}
\usepackage{amsmath}

\def\wt{\widetilde}
\def\ov{\overline}
 \def\up{\upharpoonright}
\def\cH{\mathcal H}

\def\cD{\mathcal D} \def\cR{\mathcal R} 
\def\cK{\mathcal K} \def\cL{\mathcal L} \def\cY{\mathcal Y} \def\cZ{\mathcal Z}

\def \gH{\mathfrak H}   \def \gN{\mathfrak N}

\def \bC{\mathbb C}    \def\bR{\mathbb R}
\def \l{\lambda}
\def \a{\alpha} \def \b{\beta}    \def \s{\sigma}  \def \t{\theta} \def\g {\gamma}
\def\d {\delta} \def\om {\omega} \def\Om {\Omega}
\def \f{\varphi}  \def \G{\Gamma} \def\D {\Delta}

\def \C{\widetilde {\mathcal C}}
\def \CA{\C(\cH_0,\cH_1)}
\def \CB{\C(\cH_1,\cH_0)}
\def \CG{\C(\mathcal H)}

\def \cd {\cdot}

\def \Ker{\text{Ker}}  
\def \exa { {Ext}_A}  \def \exl { {Ext}_{L_0}}  
  \def\tm{\times}
\def  \RH {\wt R (\cH_0,\cH_1)} \def \RZ {\wt R^0 (\cH_0,\cH_1)}
\def  \Rh {\wt R (\cH)} \def  \Rz {\wt R^0 (\cH)}
\def \pair {\tau=\{\tau_+,\tau_-\}}
\def\op {\{(C_0,C_1);\cK\}}

\newcommand {\LH}[1] {L'_2[#1,H ]}
  \def \RH {\wt R (\cH_0,\cH_1)}
\def \RZ {\wt R^0 (\cH_0,\cH_1)}

\def\bt{\{\cH,\G_0,\G_1\}}
\def\bta{\{\cH_0\oplus \cH_1,\Gamma _0,\Gamma _1\}}

\def \R {\mathbb R (\l)}

\newtheorem{theorem}{Theorem}[section]
\newtheorem{proposition}[theorem]{Proposition}
\newtheorem{corollary}[theorem]{Corollary}
\newtheorem{lemma}[theorem]{Lemma}

\theoremstyle{definition}
\newtheorem {definition} [theorem]{Definition}
\theoremstyle{remark}
\newtheorem{remark}[theorem]{Remark}
\numberwithin{equation}{section}

\numberwithin{equation}{section}
\begin{document}
\title[On generalized resolvents and characteristic matrices]
{On generalized resolvents and characteristic matrices of differential operators }
\author {Vadim Mogilevskii}
\begin{abstract}
The main objects of our considerations are differential operators generated by a formally selfadjoint differential expression of an even order on the interval $[0,b\rangle \;(b\leq \infty)$ with operator valued coefficients. We complement  and develop the known Shtraus' results on generalized resolvents  and characteristic matrices of the minimal operator $L_0$. Our approach is based on the concept of a decomposing boundary triplet which enables to establish a connection between the Straus' method and boundary value problems (for singular differential operators) with  a spectral parameter in a boundary condition. In particular we provide a parametrization of all characteristic matrices $\Om (\l)$ of the operator $L_0$ immediately in terms of the Nevanlinna boundary parameter $\tau(\l)$. Such  a parametrization is given  in the form of the block-matrix representation of $\Om (\l)$ as well as by means of the  formula for  $\Om (\l)$ similar to the well known Krein-Naimark  formula for generalized resolvents. In passing we obtain the representation of canonical and generalized resolvents in the form of integral operators with the operator kernel (the Green function) defined in terms of fundamental operator solutions of the equation $l[y]-\l\, y=0$.
\end{abstract}

\maketitle
\section{Introduction}
Let $\gH, H$ be Hilbert spaces, let $A$ be a symmetric densely defined operator in $\gH$ with deficiency indices $n_\pm(A)$ and let $\wt A$ be a selfadjoint (canonical or exit space) extension of $A$ acting in a Hilbert space $\wt \gH \supseteq \gH$. Moreover denote by $[H]$ the set of all bounded linear operators in $H$.

As is known the formula for generalized resolvents (in the Shtraus form)
\begin{equation}\label{0.1}
\R=(\wt A(\l)-\l)^{-1}, \quad \l\in\bC_+\cup\bC_-
\end{equation}
gives a bijective correspondence between all generalized resolvents $\R:=P_\gH(\wt A-\l)^{-1}\up\gH$ of $A$ and all holomorphic families of extensions $\wt A(\l)\supset A$ such that $Im \l\cd  Im\,\wt A(\l)\geq 0 $ and $(\wt A(\l))^*=\wt A(\ov\l)$.

Next assume that $\D=[0,b\rangle\; (b\leq\infty)$ is an interval in $\bR$ and
\begin{equation}\label{0.2}
l[y]=\sum_{k=1}^n (-1)^k ( (p_{n-k}y^{(k)})^{(k)}-\tfrac {i}{2}
[(q_{n-k}^*y^{(k)})^{(k-1)}+(q_{n-k} y^{(k-1)})^{(k)}])+p_n y
\end{equation}
is a formally selfadjoint differential expression of an even order $2n$ with the smooth enough  coefficients $p_k(\cd),q_k(\cd):\D\to \bC$. Denote by $L_0$ and $L $ minimal and maximal operators respectively induced by the expression \eqref{0.2} in the Hilbert space $\gH:=L_2(\D)$. It is known that $L_0$ is a closed densely defined symmetric operator with not necessarily equal deficiency indices $n_\pm(L_0)$ and $L_0^*=L$.

Let $Y_0(t,\l)$ be the "canonical" $2n$-component operator solution of the equation $l[y]-\l y=0$ (see \eqref{1.25}) and let $J=\text{codiag}\, (-I_{\bC^n},I_{\bC^n} )$ be a signature operator in $\bC^{n}\oplus \bC^n$. By using formula \eqref{0.1} A.V. Shtraus showed \cite{Sht57} that   each generalized resolvent $\R$ of the minimal operator $L_0$ admits the representation
\begin{equation}\label{0.3}
(\R f)(x)=\smallint\limits_0^b  G (x,t,\l)f(t)\, dt,\;\;\;\; f=f(\cd)\in\gH,
\end{equation}
where the  kernel $G(\cd, \cd, \l):\D\tm\D\to \bC$ is defined by \begin{equation}\label{0.4}
G (x,t,\l)=Y_0(x,\l)(\Om(\l)+\tfrac 1 2 \,\text{sgn} (t-x)J)Y_0^*(t,\ov\l), \quad \l\in \bC_+\cup \bC_-.
\end{equation}
Here $\Om(\cd):\bC_+\cup \bC_-\to [\bC^n\oplus \bC^n]$ is a Nevanlinna operator (matrix) function, which is called a characteristic matrix of the corresponding generalized resolvent $\R$\cite{Sht57}. Next by using the matrix $\Om(\cd)$ A.V. Shtraus constructed in \cite{Sht57} all spectral functions (both orthogonal and not orthogonal) of the operator $L_0$. Observe also that in \cite{Bru74} these results were extended to differential expressions \eqref{0.2} with operator valued coefficients.

A rather different approach in the theory of  spectral functions is based on the application  of boundary value problems with  a spectral parameter in a boundary condition (see \cite{KacKre,DL96,BinBro02} and the references therein). Namely let
\begin{equation}\label{0.4a}
l[y]=-y''+q(t)y, \qquad  (q(t)=\ov{q(t)}), \;\;\;t\in [0,\infty)
\end{equation}
be a Sturm-Liouville expression such that $n_\pm(L_0)=1$ (the limit point case at $\infty$), let $m(\l)$ be the corresponding Titchmarsh-Weyl function \cite{CodLev,Nai} and let $\R$ be a generalized resolvent of the operator $L_0$. Then formula \eqref{0.1} immediately implies that there exists a Nevanlinna function $\tau (\l)$ such that for each $f\in\gH$ the function $y=\R f$ is the unique solution of a boundary value problem
\begin{align}
l[y]-\l y=f\qquad\qquad\qquad\label{0.5}\\
y(0)-\tau (\l)y'(0)=0,\;\;\;\l\in\bC_+\cup\bC_- \label{0.6}
\end{align}
Moreover as was shown in \cite{DL96} the characteristic matrix $\Om(\cd)$ of the generalized resolvent $\R$ is associated with the Nevanlinna parameter $\tau(\cd)$ in \eqref{0.6} via
\begin{equation}\label{0.7}
\Om(\l)=\frac 1 {m(\l)+\tau (\l)}\begin{pmatrix} m(\l)\tau(\l) &  m(\l)\cr m(\l) & -1 \end{pmatrix}.
\end{equation}
In this connection note that the case $n_\pm(L_0)=2$ is more complicated, since in this case the only boundary condition \eqref{0.6} at the point $0$ is not sufficient to define a selfadjoint extension of the operator $L_0$ and, consequently, a spectral function. Moreover, the situation becomes essentially more complicated for a singular differential expression of a higher order.

In the present paper we complement and develop  the above results on generalized resolvents and characteristic matrices. The main objects here are differential operators in the Hilbert space $\gH:=L_2(\D;H)$ generated by a formally  selfadjoint differential expression \eqref{0.2} with operator valued coefficients $p_k(\cd), q_k(\cd):\D\to [H]\; (\dim H\leq\infty)$. Denote by $n_{b\pm}$ deficiency indices of the expression $l[y]$ at the end $b$ of the interval $\D=[0,b\rangle$ \cite{Mog}. Our method is based on the concept of a decomposing $D$-triplet for $L$ introduced in \cite{Mog} for the case $n_{b-}\leq n_{b+}$. To simplify further considerations we suppose below that $n_{b-}= n_{b+}$ (and, therefore, $n_+(L_0)=n_-(L_0)$). In this case a decomposing $D$-triplet can be replaced with a decomposing boundary triplet for $L$, which is defined as follows.

Denote by $\cD$ the domain of the maximal operator $L$. Let  $\cH'$ be a Hilbert space and let $\G_j':\cD\to \cH',\;j\in\{0,1\}$ be linear maps. Construct linear operators $\G_j:\cD\to H^n\oplus \cH_j,\;j\in\{0,1\} $ by
\begin{equation}\label{0.8}
\G_0 y=\{y^{(2)}(0),  \G'_0y\}\,(\in H^n\oplus \cH'), \;\;\; \G_1 y=\{-y^{(1)}(0),  \G'_1y\}\,(\in H^n\oplus \cH'), \quad y\in\cD.
\end{equation}
(here $y^{(1)}(0),\,y^{(2)}(0)\in H^n$ are vectors of the quasi-derivatives at the point $0$, see \eqref{1.21}). Then a collection $\Pi=\bt$, where $\cH:=H^n\oplus\cH'$ and $\G_0,\G_1$ are linear maps \eqref{0.8}, is called a decomposing boundary triplet for $L$ if the map $\G=(\G_0\;\;\G_1)^\top$ is surjective and the following identity holds
\begin{equation*}
(Ly,z)-(y,Lz)=(\G_1 y,\G_0 z)- (\G_0 y,\G_1 z),\quad y,z \in \cD.
\end{equation*}
The last two conditions mean that $\Pi$ is a boundary triplet for $L$ in the sense of \cite{GorGor}. Therefore according to \cite{DM91} the equality
\begin{equation*}
\G_1 f_\l=M(\l)\,\G_0 f_\l, \quad f_\l\in\gN_\l(L_0):=\Ker (L-\l), \;\;\l\in\bC_+\cup\bC_-
\end{equation*}
defines a uniformly strict Nevanlinna operator function $M(\l)(\in [\cH])$ which is called the Weyl function corresponding to the triplet $\Pi$. In our case the function $M(\cd)$ admits the block matrix representation
\begin{equation}\label{0.9}
M(\l)=\begin{pmatrix} m(\l) & M_{2}(\l)\cr M_{3}(\l) & M_{4}(\l)\end{pmatrix}:
H^n\oplus\cH' \to H^n\oplus \cH', \quad \l\in\bC_+\cup\bC_-.
\end{equation}
Formula \eqref{0.9} induces the uniformly strict Nevanlinna function $m(\l)(\in [H^n])$ which is called the $m$-function of the operator $L_0$. In the scalar case ($\dim H=1$) this function coincides with the classical characteristic (Titchmarsh-Weyl) function for decomposing boundary conditions (for more details see \cite{Mog}).

Next denote by $\wt R(\cH)$ the set of all Nevanlinna operator pairs (Nevanlinna families of linear relations) \cite{DM00}
\begin{equation}\label{0.10}
\tau (\l)=\{(C_0(\l),C_1(\l));\cH\}, \quad \l\in\bC_+\cup \bC_-
\end{equation}
with operator functions $C_0(\l)$ and $C_1(\l)$ defined by the block-matrix representations
\begin{align}
C_0(\l)=(\hat C_{2}(\l)\;\;C'_{0}(\l)): H^n\oplus\cH'\to\cH, \qquad \qquad\qquad \qquad\qquad \qquad\qquad  \label{0.11}\\
\qquad \qquad\qquad \qquad \qquad \qquad\qquad  C_1(\l)=(\hat C_{1}(\l)\;\;C'_1(\l)):H^n\oplus \cH'\to\cH .\nonumber
\end{align}
For a given $f\in\gH$ and $\tau(\l)\in\Rh$ consider a boundary value problem
\begin{align}
l[y]-\l y=f\qquad \qquad\qquad\qquad\qquad\label{0.12}\\
\hat C_{1}(\l)y^{(1)}(0)+\hat C_{2}(\l)y^{(2)}(0)+C'_{0}(\l)\G_0' y-C'_{1} (\l)\G_1' y=0.\label{0.13}
\end{align}
A function $y:\D\to H$ is called a solution of the problem \eqref{0.12}, \eqref{0.13} if it belongs to $\cD$ and obeys the equation \eqref{0.12} and the relation \eqref{0.13}.

For a regular expression $l[y]$ (defined on a finite segment $\D=[0,b]$) one can put $\G_0'y=y^{(2)}(b), \G_1'y=y^{(1)}(b)$ \cite{Rof69,HolRof}, in which case the boundary condition \eqref{0.13} takes the "classical"  form (c.f. \cite{KacKre}). For a singular expression $l[y]$ the boundary operators $\G_0'$ and $\G_1'$ can by explicitly defined in terms of limits  of some regularizations of quasi-derivatives $y^{[j]}(t)$ at the point $b$ (see Proposition 3.10 in \cite{Mog}). Hence the relation \eqref{0.13} is a boundary condition given in terms of boundary values $y^{(1)}(0), y^{(2)}(0)$ (at the regular end $0$) and $\G_0' y, \G_1' y$ (at the singular end $b$) of a function $y\in\cD$. Moreover, an application  of V.M. Bruk's results \cite{Bru76} to the decomposing boundary triplet \eqref{0.8} gives the following theorem.
\begin{theorem}\label{th0.1}
For every generalized resolvent $\R$ of the minimal operator $L_0$ there exists the unique operator pair $\tau(\l)\in\Rh$ such that for every $f\in\gH$ the function $y=\R f$ is the unique solution of the boundary value problem \eqref{0.12}, \eqref{0.13}, and, conversely, for each $\tau (\l) \in\Rh$ the relations \eqref{0.12}, \eqref{0.13} define a generalized resolvent of $L_0$.
\end{theorem}
In view of Theorem \ref{th0.1} formulas \eqref{0.12}, \eqref{0.13} parameterize all generalized resolvents $\R$ by means of a Nevanlinna boundary parameter $\tau(\l)$. In the sequel this fact will be written as $\R=\bR_\tau(\l)$.

The main result of the paper is a description of all characteristic matrices $\Om(\cd)$ of the operator $L_0$ immediately in terms of the boundary parameter $\tau(\l)\in\Rh$. Namely, with every Nevanlinna operator pair $\tau(\l)$ we associate an operator function $\wt\Om_\tau (\cd):\bC_+\cup\bC_-\to [\cH\oplus\cH]$, given by the block-matrix representation
\begin{equation}\label{0.14}
\wt\Om_{\tau }(\l)=\begin{pmatrix} M(\l)-M(\l)(\tau(\l) +M(\l))^{-1}M(\l) & -\tfrac 1 2 I_{\cH} +M(\l) (\tau(\l) +M(\l))^{-1} \cr -\tfrac 1 2 I_{\cH} +(\tau(\l) +M(\l))^{-1}M(\l) & -(\tau(\l) +M(\l))^{-1}\end{pmatrix}
\end{equation}
Next, the inclusion $H^n\subset\cH(=H^n\oplus\cH')$ implies that $H^n\oplus H^n\subset \cH\oplus\cH$. This enables to introduce the operator function $\Om_\tau (\cd):\bC_+\cup\bC_-\to [H^n\oplus H^n]$ by
\begin{equation}\label{0.15}
\Om_\tau (\l)= P_{H^n\oplus H^n}\,\wt \Om_{\tau }(\l)\up H^n\oplus H^n, \;\; \l\in\bC_+\cup\bC_-.
\end{equation}
We show in the paper that for every $\tau(\l)\in\Rh$ the Shtraus-Bruk characteristic matrix $\Om(\cd)$ of the corresponding generalized resolvent $\bR_\tau(\l)$ obeys $\Om(\l)=\Om_\tau(\l)$. Hence the equalities \eqref{0.14} and \eqref{0.15} give a parametrization of all characteristic matrices $\Om(\cd)$ in terms of a boundary parameter $\tau(\cd)$ associated with the boundary value problem \eqref{0.12}, \eqref{0.13}. Moreover, the following theorem shows that the same parametrization can be given in the form similar to that of the well-known Krein-Naimark formula for generalized resolvents \cite{KreLan71}.
\begin{theorem}\label{th0.2}
Let $A_0$ be  a selfadjoint extension of $L_0$ with the domain $\cD (A_0)=\{y\in\cD: y^{(2)}(0)=0, \,\G_0'y=0\}$ and let
\begin{equation*}
S(\l)=\begin{pmatrix} -m(\l)&
-M_{2}(\l) \cr I_{H^n} & 0 \end{pmatrix}:H^n\oplus\cH'\to H^n\oplus H^n,\;\;\;\l\in\bC_+\cup\bC_-,
\end{equation*}
where $m(\l)$ and $M_2(\l)$ are taken from \eqref{0.9}. Then the characteristic matrix of the resolvent $(A_0-\l)^{-1}$ is  $\Om_0(\l)=\begin{pmatrix} m(\l) & -\tfrac 1 2 I_{H^n} \cr -\tfrac 1 2 I_{H^n} & 0 \end{pmatrix}$ and the equality
\begin{equation}\label{0.16}
\Om(\l)(=\Om_\tau(\l))=\Om_0(\l)-S(\l)(\tau(\l)+M(\l))^{-1}S^*(\ov\l), \quad\l\in\bC_+\cup\bC_-
\end{equation}
gives a bijective correspondence between all characteristic matrices $\Om(\l)$ of the operator $L_0$ and all Nevanlinna operator pairs $\tau(\cd)\in\Rh$. Moreover  $\Om(\l)$ is a characteristic matrix of a  canonical resolvent if and only if $\tau(\l)\equiv \tau=\tau^*\; (\l\in\bC_+\cup\bC_-)$.
\end{theorem}
Thus Theorem \ref{th0.2} together with formulas \eqref{0.3} and \eqref{0.4} establishes a connection between  the boundary problem \eqref{0.12}, \eqref{0.13} and the Shtraus' method in the theory of differential operators.

Next by using \eqref{0.16} we prove the inequality
\begin{equation}\label{0.17}
(Im \l)^{-1}\, Im\, \Om_\tau(\l)  \geq  \int_0^b  U_\tau^*(t,\l) U_\tau(t,\l)\, dt, \quad \l\in\bC_+\cup\bC_-,
\end{equation}
which turns into the equality for canonical (orthogonal) characteristic matrices $\Om_\tau(\cd)$. In formula \eqref{0.17} $U_\tau(t,\l)$ is the operator solution of the equation \eqref{0.12}, which is defined by some initial conditions written in terms of the operator functions \eqref{0.11}. The inequality \eqref{0.17} immediately implies that $\Om_\tau(\cd)$ is a Nevanlinna function (c.f. proof of this fact in \cite{Sht57,Bru74}). Observe also that formula \eqref{0.17} is similar to that obtained in \cite{Ber,Gor66} for different classes of boundary value problems.

In the final part of the paper we consider the simplest situation
\begin{equation}\label{0.18}
n_{b+}=n_{b-}=0,
\end{equation}
which in the case $\dim H<\infty$ is equivalent to the relation $n_+(L_0)=n_-(L_0)=n\,\dim H$ (minimality of the deficiency indices). It turns out that under the condition \eqref{0.18} the unique decomposing boundary triplet for $L$ is $\Pi=\{H^n,\G_0,\G_1\}$ with $\G_0 y=y^{(2)}(0), \; \G_1 y=-y^{(1)}(0),\; y\in\cD$. Moreover by formulas \eqref{0.9} and \eqref{0.14}, \eqref{0.15} $M(\l)=m(\l)$ and the characteristic matrix $\Om_\tau (\l)$ is defined by \eqref{0.14} (with $\Om_\tau (\l)$ in place of $\wt\Om_\tau (\l)$ ). The last statement implies that the equality \eqref{0.7} for the Sturm-Liouville operator is a particular case of the general formula \eqref{0.14}. Observe also that the form of the matrices \eqref{0.7} and \eqref{0.14} goes back to I.S. Kac \cite{Kac63}.

In passing with the above main statements  we obtain some results which, to our mind, are of a self-contained interest. In particular we complement and generalize the results from \cite{RofHol84,HolRof} concerning operator fundamental solutions of the equation $l[y]-\l y=0$ and spectrum of differential operators (see Theorems \ref{th1.9a} and \ref{th2.3}). Moreover, in Theorems \ref{th2.8} and \ref{th3.3}  for the case $\dim H\leq\infty$  canonical and generalized resolvents are given in the form  of integral operators with the operator kernel (the Green function) defined immediately in terms of fundamental solutions. We suppose that such a representation of the Green function is new even in the scalar case for an operator $L_0$ with equal intermediate deficiency indices $n<n_+(L_0)=n_-(L_0)<2n$.

In conclusion note that all specified results are actually obtained for differential operators with arbitrary (possibly unequal) deficiency indices by using the method of decomposing $D$-triplets. In this way the known technique of boundary triplets \cite{GorGor,DM91,DM92} becomes slightly more complicated.

\section{Preliminaries}
\subsection{Notations}
The following notations will be used throughout the paper: $\gH$,
$\cH$ denote Hilbert spaces; $[\cH_1,\cH_2]$  is the set of all
bounded linear operators defined on $\cH_1$ with values in $\cH_2$; $[\cH]:=[\cH,\cH]$; $A\up \cL$ is the restriction of an operator $A$ onto the linear manifold $\cL$; $P_\cL$ is the
orthogonal projector in $\gH$ onto the subspace $\cL\subset\gH$;
$\bC_+\,(\bC_-)$ is the upper (lower) half-plain  of the complex
plain.

Recall that a closed linear relation from $\cH_0$ to $\cH_1$ is a closed subspace in $\cH_0 \oplus \cH_1$. The set of all closed linear
relations from $\cH_0$ to $\cH_1$ (from $\cH$ to $\cH$) will be denoted by  $\C
(\cH_0,\cH_1)$ ($\C(\cH)$). A closed linear operator $T$ from $\cH_0$ to $\cH_1$  is identified  with its graph $\text {gr} T\in\CA$.

For a  relation $T \in \C (\cH_0,\cH_1)$ we denote by $\cD(T),\,\cR (T)$ and $\text {Ker}T$  the domain,  range and the kernel  respectively. The inverse $T^{-1}$ and adjoint  $T^*$ are  relations  defined by
\begin{align*}
T^{-1}=\{\{f^{\prime},f\}:\{f,f^{\prime}\}\in T\},\quad
T^{-1}\in \C (\cH_1,\cH_0)\qquad\qquad\qquad\\
T^*=\{\{g,g^{\prime}\}\in \cH_1\oplus \cH_0:(f^{\prime},g)=(f,g^
{\prime}),\ \ \{f,f^{\prime}\}\in T\}, \quad T^*\in
\C(\cH_1,\cH_0).
\end{align*}
In the case $T\in\CA$ we write:

$0\in \rho (T)$\ \ if\ \ $\Ker T=\{0\}$\  and\  $\cR (T)=\cH_1$, or equivalently if $T^{-1}\in [\cH_1,\cH_0]$;

$0\in \hat\rho (T)$\ \ if\ \  $\Ker T=\{0\}$\ and\   $\cR (T)$ is a closed subspace in $\cH_1$;

$0\in \sigma_c (T)$\ \  if\ \  $\Ker T=\{0\}$ and $\ov{\cR (T)}=\cH_1\neq\cR
(T);$

$0\in \sigma_p(T)$\ \  if\ \  $\Ker T\neq\{0\};$\quad $0\in \sigma_r(T)$ \ \ if
\ \ $\Ker T=\{0\}$ and $\ov {\cR (T)}\neq\cH_1$.

For a linear relation $T\in \C(\cH)$ we denote by $\rho (T)=\{\l \in \bC:\ 0\in
\rho (T-\l)\}$\ and $\hat\rho (T)=\{\l \in \bC:\ 0\in \hat\rho (T-\l)\}$ the\textbf{}
resolvent set and the set of regular type points of $T$ respectively. Next,
$\sigma (T)=\bC \backslash \rho (T)$ stands for the spectrum of $T.$ The
spectrum  $\sigma (T)$ admits the following classification:
$\sigma_c(T)=\{\l \in \bC:0\in \sigma_c (T-\l)\}$ is the continuous spectrum;  $\sigma_p (T)=\{\l \in \bC:0\in \sigma_p (T-\l)\}$ is the point spectrum; $\sigma_r (T)=\sigma (T)\setminus (\sigma_p (T)\cup\sigma_c (T))=\{\l \in \bC:0\in \sigma_r (T-\l)\}$ is the residual spectrum.

Let $T\in\C (\cH)$ be a densely defined operator. For any $\l\in\bC$  we put
\begin{equation*}
\gN_\l (T):=\Ker (T^*-\l)\,(=\cH\ominus \cR (T-\ov\l)).
\end{equation*}
If $\ov\l\in\hat\rho (T)$, then $\gN_\l (T)$ is a defect subspace of the operator $T$.

In what follows we will use the operators $J_{\cH_0,\cH_1}\in [\cH_0\oplus\cH_1,\cH_1\oplus\cH_0]$ and $J_\cH\in [\cH\oplus\cH]$ defined by
\begin{equation}\label{1.0.0}
J_{\cH_0,\cH_1}=\begin{pmatrix}0 & -I_{\cH_1} \cr I_{\cH_0} & 0 \end{pmatrix}:\cH_0\oplus\cH_1\to\cH_1\oplus\cH_0 \;\;\;\text {and}\;\;\; J_\cH=J_{\cH,\cH}.
\end{equation}
\subsection{Holomorphic operator pairs }
Let $\Lambda$ be an open set in $\bC$, let $\cK,\cH_0,\cH_1$  be Hilbert spaces and let $K_j(\cd):\Lambda \to [\cK,\cH_j], \; j\in \{0,1\}$ be a pair of holomorphic operator functions. In what follows we identify such a pair with a holomorphic operator function
\begin{equation}\label{1.0}
K(\l)=(K_0(\l)\;\;\; K_1(\l))^\top :\cK\to \cH_0\oplus\cH_1, \quad \l\in \Lambda.
\end{equation}
Similarly a pair of holomorphic  operator functions $C_j(\cd):\Lambda \to [\cH_j,\cK], \; j\in \{0,1\}$ will be identified with an  operator  function
\begin{equation}\label{1.1}
C(\l)=(C_0(\l)\;\;\; C_1(\l)): \cH_0\oplus\cH_1 \to\cK, \quad \l\in \Lambda.
\end{equation}

A pair \eqref{1.0} will be called closed if $\ov{K(\l)\cK}=K(\l)\cK, \;\l\in\Lambda $. Moreover a pair \eqref{1.0} (\eqref{1.1}) will be called admissible if $\Ker K(\l)=\{0\}$ (respectively $\cR (C(\l))=\cK$) for all $\l\in \Lambda$. In the sequel  all pairs  \eqref{1.0} and \eqref{1.1} are admissible unless otherwise stated.
\begin{definition}\label{def1.0}
Two pairs of holomorphic operator functions
\begin{equation*}
K^{(j)} (\l)=\begin{pmatrix} K_0^{(j)}(\l) \cr K_1^{(j)}(\l) \end{pmatrix}: \cK_j \to \cH_0\oplus\cH_1 \;\;\;\;\;  \bigl( C^{(j)}(\l)=( C_0^{(j)}(\l) \;\;\; C_1^{(j)}(\l)): \cH_0\oplus\cH_1 \to\cK_j \bigr),
\end{equation*}
where $j\in\{1,2\}$ and  $\l\in \Lambda$, are said to be equivalent if $K^{(1)} (\l)=K^{(2)} (\l)\f (\l)$ (respectively  $C^{(2)} (\l)=\f (\l) C^{(1)} (\l)$) with a holomorphic isomorphism $\f(\cd): \Lambda\to [\cK_1, \cK_2]$.
\end{definition}
It is clear that the set of all operator pairs \eqref{1.0} (respectively \eqref{1.1}) falls into nonintersecting classes of equivalent pairs.
\begin{definition}\label{def1.0.1}
A function $\tau(\cd):\Lambda\to \CA$ is called holomorphic on $\Lambda$ if there exists a holomorphic operator pair \eqref{1.0} such that $\tau (\l)=K(\l)\cK$ or equivalently
\begin{equation}\label{1.1a}
\tau (\l)=\{K_0(\l), K_1(\l);\cK\}:=\{\{K_0(\l)h,K_1(\l)h\}:h\in\cK\}, \quad \l\in\Lambda.
\end{equation}
\end{definition}
The relation \eqref{1.1a} establishes a bijective correspondence between all holomorphic functions $\tau(\cd):\Lambda\to \CA$ and all equivalence classes of closed holomorphic pairs \eqref{1.0}. Therefore we will identify (by means of \eqref{1.1a}) a holomorphic $\CA$-valued function $\tau (\cd)$ and the corresponding equivalence class of (closed) holomorphic  pairs \eqref{1.0}. Similarly we will identify an equivalence class of  holomorphic operator  pairs \eqref{1.1} and a function $\tau(\cd):\Lambda\to \CA$ given for all $\l\in\Lambda$ by
\begin{equation}\label{1.1b}
\tau (\l)=\{(C_0(\l), C_1(\l));\cK\}:=\{\{h_0,h_1\}\in \cH_0\oplus\cH_1: C_0(\l)h_0+C_1(\l)h_1=0\} .
\end{equation}
It follows from the Liouville's theorem that in the case $\Lambda=\ov\bC$ pairs of  holomorphic operator functions \eqref{1.0} and \eqref{1.1} turn into pairs of operators
\begin{align}
K=(K_0\;\;\; K_1)^\top :\cK\to \cH_0\oplus\cH_1\label{1.2},\\
C=(C_0\;\;\; C_1): \cH_0\oplus\cH_1 \to\cK.\label{1.3}
\end{align}
Moreover the equivalence of such pairs (in the sense of Definition \ref{def1.0}) is realized by a (constant) isomorphism $\f\in [\cK_1,\cK_2]$. Hence $\tau (\l)\equiv\t,\;\l\in\ov\bC$ and \eqref{1.1a} takes the form
\begin{equation}\label{1.4a}
\t=\{K_0,K_1;\cK\}:=\{\{K_0h,K_1h\}:h\in\cK\}.
\end{equation}
Formula \eqref{1.4a} gives a bijective correspondence between all linear relations $\t\in\CA$ and all equivalence classes of closed operator pairs \eqref{1.2}. Similarly the relation
\begin{equation}\label{1.4}
\t=\{(C_0,C_1);\cK\}:=\{\{h_0, h_1\}\in\cH_0\oplus\cH_1:\, C_0h_0+C_1h_1=0\}
\end{equation}
gives a bijective correspondence between all  $\t\in\CA$ and all equivalence classes of  operator pairs \eqref{1.3}. Therefore we will denote by $\CA$ both the set of all closed linear relations  from $\cH_0$ to $\cH_1$ and the set of  all equivalence classes of operator pairs \eqref{1.2} (\eqref{1.3}), identifying them by means of \eqref{1.4a}(\eqref{1.4} respectively).

The following lemma is immediate from Lemma 2.1 in \cite{MalMog02}.
\begin{lemma}\label{lem0.1}
Let $\t=\op\in \CA$, let $\t^*=\{(C_{1*},C_{0*});\cK_*\}\in \CB$ be the adjoint operator pair (linear relation) and let $B\in [\cH_1,\cH_0]$. Then $0\in \rho (C_{1*}+C_{0*}B)\iff 0\in \rho (C_0^*+B C_1^*)$ and the following equality holds
\begin{equation*}
C_1^* (C_0^*+B C_1^*)^{-1}=(C_{1*}+C_{0*}B)^{-1}C_{0*}(=- (\tau^*-B)^{-1}).
\end{equation*}
\end{lemma}
Next recall some results and definitions from our paper \cite{Mog06.1}.

Let $\cH_1$ be a subspace
in a Hilbert space $\cH_0$, let $\cH_2:=\cH_0\ominus\cH_1$ and let $P_j$ be the orthoprojector in $\cH_0$ onto $\cH_j,\; j\in\{1,2\}$. With every linear relation $\t\in\CA$ we associate a $\times$-adjoint linear relation $\t^\times\in\CA$, which is defined as the set of all vectors $\hat k=\{k_0,k_1\}\in \cH_0\oplus\cH_1$ such that
\begin{equation*}
(k_1,h_0)-(k_0,h_1)+i(P_2 k_0,P_2 h_0)=0,\quad \{h_0,h_1\}\in\t
\end{equation*}

Using this definition and the correspondence \eqref{1.4} we introduce the notion of a $\times$-adjoint operator pair (or more precisely a class of $\times$-adjoint  operator pairs) $\t^\times=\{(C_{0\times}, C_{1\times});\cK_\times\}\in\CA$ corresponding to an operator pair $\t=\{(C_0,C_1);\cK\}\in\CA$. As was shown in \cite{Mog06.1} $(\t^\times)^\times=\t$. Moreover Proposition 3.1 in \cite{Mog06.1} implies that the adjoint operator pair $\t^*$ admits the representation $\t^*=\{(C_{1*},C_{0*});\cK_\times\}$ with
\begin{equation}\label{1.5}
C_{1*}=C_{0\times}\up\cH_1, \qquad  C_{0*}= C_{1\times}P_1-i C_{0\times}P_2.
\end{equation}
It follows from  \eqref{1.5} that $\t^\times=\t^*$  in the case $\cH_0=\cH_1:=\cH$.
\begin{definition}\label{def1.0.2}\cite{Mog06.1}
A linear relation $\t\in\CA$ belongs to the class $Dis (\cH_0, \cH_1)$ ($Ac (\cH_0, \cH_1) $) if $\f_\t (\hat h):=2 \, Im (h_1,h_0)+||P_2 h_0||^2\geq 0$ (respectively $\f_\t (\hat h)\leq 0$) for all $\hat h =\{h_0, h_1\} \in\t$ and there are not extensions $\wt \t\supset\t,\;\wt\t\neq\t$ with the same property.
\end{definition}
Note that in the case $\cH_0=\cH_1=:\cH$ the class  $Dis (\cH, \cH)$ ($Ac (\cH, \cH) $) coincides with the set of all maximal dissipative (accumulative) linear relations in $\cH$.

Next assume that $\cK_+'$ and $\cK_-'$ are Hilbert spaces and
\begin{equation}\label{1.6}
\tau_+(\l)=\{K_0(\l),K_1(\l);\cK_+'\}, \;\;\l\in\bC_+; \quad \tau_-(\l)=\{N_0(\l),N_1(\l);\cK_-'\}, \;\;\l\in\bC_-
\end{equation}
are equivalence classes of holomorphic operator pairs
\begin{equation*}
(K_0(\l)\;\;K_1(\l))^\top:\cK_+'\to \cH_0\oplus\cH_1, \;\;\l\in\bC_+;\;\;\; (N_0(\l)\;\;N_1(\l))^\top:\cK_-'\to \cH_0\oplus\cH_1, \;\;\l\in\bC_-
\end{equation*}
with the block-matrix representations
\begin{equation*}
K_0(\l)=\begin{pmatrix} K_{01}(\l)\cr K_{02}(\l)\end{pmatrix}:\cK_+'\to \cH_1\oplus\cH_2, \;\;\;  N_0(\l)=\begin{pmatrix}N_{01}(\l)\cr N_{02}(\l)\end{pmatrix}:\cK_-'\to \cH_1\oplus\cH_2.
\end{equation*}
Similarly let  $\cK_+$ and $\cK_-$ be  Hilbert spaces and let
\begin{equation}\label{1.7}
\tau_+(\l)=\{(C_0(\l),C_1(\l));\cK_+\}, \;\;\l\in\bC_+; \;\;\;\; \tau_-(\l)=\{(D_0(\l),D_1(\l));\cK_-\}, \;\;\l\in\bC_-
\end{equation}
be equivalence classes of holomorphic operator pairs
\begin{equation*}
(C_0(\l)\;\;C_1(\l)):\cH_0\oplus\cH_1\to\cK_+ , \;\;\l\in\bC_+;\;\; (D_0(\l)\;\;D_1(\l)):\cH_0\oplus\cH_1\to\cK_-, \;\;\l\in\bC_-
\end{equation*}
with the block-matrix representations
\begin{equation*}
C_0(\l)=(C_{01}(\l)\;\;C_{02}(\l)):\cH_1\oplus\cH_2\to\cK_+ , \;\;\;  D_0(\l)=(D_{01}(\l)\;\;D_{02}(\l)):\cH_1\oplus\cH_2\to\cK_-.
\end{equation*}
\begin{definition}\label{def1.0.3}
1) A collection $\pair$ of two holomorphic  operator pairs \eqref{1.6} (or more precisely of two equivalence classes of operator pairs  \eqref{1.6}) belongs to the Nevanlinna type class $\RH$ if it obeys the relations
\begin{align}
2\,Im(K_{01}^*(\l)K_1(\l))-K_{02}^*(\l)K_{02}(\l)\geq 0, \quad 0\in\rho (K_1(\l)+iK_{01}(\l)), \;\;\; \l\in\bC_+\label{1.7a}\\
2\,Im(N_{01}^*(\l)N_1(\l))-N_{02}^*(\l)N_{02}(\l)\leq 0, \quad 0\in\rho (N_1(\l)-iN_0(\l)), \;\; \l\in\bC_-\quad\label{1.7c}\\
N_1^*(\ov\l)K_{01}(\l)-N_{01}^*(\ov\l)K_{1}(\l)+iN_{02}^*(\ov\l)K_{02} (\l)=0, \;\;\; \l\in\bC_+.\qquad\qquad\label{1.7d}
\end{align}

2) A collection $\pair$ of two holomorphic  operator pairs \eqref{1.7} is referred to the class $\RH$ if
\begin{align}
2\,Im(C_{1}(\l)C_{01}^*(\l))+C_{02}(\l)C_{02}^*(\l)\geq 0, \quad 0\in\rho (C_0(\l)-iC_{1}(\l)P_1), \;\;\; \l\in\bC_+\label{1.7e}\\
2\,Im(D_1(\l)D_{01}^*(\l))+D_{02}(\l)D_{02}^*(\l)\leq 0, \quad 0\in\rho (D_{01}(\l)+iD_1(\l)), \;\;\; \l\in\bC_-\label{1.7f}\\
C_1(\l)D_{01}^*(\ov\l)-C_{01}(\l)D_{1}^*(\ov\l)+iC_{02}(\l)D_{02}^* (\ov\l)=0, \;\;\; \l\in\bC_+.\qquad\qquad\label{1.7g}
\end{align}
\end{definition}
Note that the above definition of the class $\RH$ slightly differs from that contained in \cite{Mog06.1}. At the same time by methods similar to \cite{Mog06.1}, Proposition 4.3 one can prove that \eqref{1.6} defines a pair $\pair\in\RH$ if and only if $\tau_\pm(\cd):\bC_\pm\to \CA$ are holomorphic functions obeying $-\tau_+(\l)\in Ac(\cH_0,\cH_1), \; \l\in\bC_+$ and $-\tau_-(\l)=(-\tau_+(\ov\l))^\tm, \; \l\in \bC_-$. This and \cite{Mog06.1}, Proposition 3.1, 4) imply that collections \eqref{1.6} and \eqref{1.7} are associated by $\cK_\pm=\cK'_\mp$ and
\begin{align}
C_0(\l)=(N_1^*(\ov\l)\;\;iN_{02}^*(\ov\l)):\cH_1\oplus\cH_2\to\cK'_-, \quad C_1(\l)=-N_{01}^*(\ov\l), \quad \l\in\bC_+ \label{1.8}\\
D_0(\l)=(K_1^*(\ov\l)\;\;iK_{02}^*(\ov\l)):\cH_1\oplus\cH_2\to\cK'_+, \quad D_1(\l)=-K_{01}^*(\ov\l), \quad \l\in\bC_- \label{1.8a}
\end{align}
define the same functions $\tau_\pm (\cd)$. Hence the relations \eqref{1.8} and \eqref{1.8a} give a bijective correspondence between collections \eqref{1.6} and \eqref{1.7}, which makes it possible to identify such collections (connected by \eqref{1.8} and \eqref{1.8a}).

One can prove that for a pair $\pair\in\RH$ the following statements are equivalent:

(a) $2\,Im(K_{01}^*(\l)K_1(\l))-K_{02}^*(\l)K_{02}(\l)= 0$ and $0\in\rho (K_1(\l)-iK_{0}(\l))$ for some  $\l\in\bC_+$;

(b) $2\,Im(N_{01}^*(\l)N_1(\l))-N_{02}^*(\l)N_{02}(\l)= 0$ and $0\in\rho (N_1(\l)+iN_{01}(\l))$ for some  $\l\in\bC_-$;

(c) $2\,Im(C_{1}(\l)C_{01}^*(\l))+C_{02}(\l)C_{02}^*(\l)= 0$ and $0\in\rho (C_{01}(\l)+iC_{1}(\l))$ for some  $\l\in\bC_+$;

(d) $2\,Im(D_1(\l)D_{01}^*(\l))+D_{02}(\l)D_{02}^*(\l)= 0$ and $0\in\rho (D_{0}(\l)-iD_{1}(\l)P_1)$ for some  $\l\in\bC_-$;

(e) the relations (a) -- (d) hold for all $\l$ in the corresponding half-planes.
\begin{definition}\label{def1.0.4}
A collection $\pair\in\RH$ is referred to the class $\RZ$ if it obeys at least one (and hence each) of the conditions (a) -- (e).
\end{definition}
It is easy to see that a collection $\pair$ of two holomorphic operator pairs \eqref{1.6} (or, equivalently, \eqref{1.7}) belongs to the class $\RZ$ if and only if $\tau_+(\l)\equiv \tau_-(\l)\equiv \t$, where $\t\in\CA$ and $(-\t)^\tm =-\t$. Therefore a pair $\pair\in\RZ$ admits the constant-valued representation
\begin{equation*}
\tau_\pm(\l)=\{K_0,K_1;\cK'\}=\{(C_0,C_1);\cK\}(=\t\in\CA), \quad \l\in\bC_\pm.
\end{equation*}
Moreover the set $\RZ$ is not empty if and only if $\dim\cH_0= \dim\cH_1$.
\begin{remark}\label{rem1.0.5}
1) Let  \eqref{1.6} be a representation of a pair $\pair\in\RH$. Then $\dim\cK'_+= \dim\cH_1, \; \dim\cK'_-= \dim\cH_0$ and, therefore, Hilbert spaces $\cK'_+$ and $\cK'_-$ can be chosen equal only in the case $\dim\cH_1=\dim\cH_0$. This explains the presence of different spaces $\cK'_+$ and $\cK_-'$ in \eqref{1.6} (respectively $\cK_+$ and $\cK_-$ in \eqref{1.7}). Observe also that in the case $\cK_+=\cK_-=:\cK$ a collection $\{\tau_+(\l), \tau_-(\l)\}$ of two operator pairs \eqref{1.7} can be considered as the unique holomorphic operator pair defined on $\bC_+\cup\bC_-$.

2) In the case $\cH_1=\cH_0=:\cH$ the class $\wt R(\cH):=\wt R (\cH,\cH)$ coincides with the well known class of Nevanlinna functions (holomorphic operator pairs) with values in $\CG$ (see for instance \cite{KreLan71,DM00}). More precisely, the equality
\begin{equation*}
\tau(\l)=\begin{cases} \tau_+(\l), \;\;\l\in\bC_+ \cr \tau_-(\l), \;\;\l\in\bC_-\end{cases}
\end{equation*}
gives a bijective correspondence between all collections $\pair\in \Rh$ and all Nevanlinna functions $\tau(\cd):\bC_+\cup\bC_-\to\CG$.

Letting in \eqref{1.7} $\cK_+=\cK_-=:\cK$ and $C_j(\l):=D_j(\l),\;\l\in\bC_-,\; j\in\{0,1\}$ one can represent
a Nevanlinna function $\tau(\cd)\in\Rh$ as
\begin{equation}\label{1.9}
\tau(\l)=\{(C_0(\l),C_1(\l));\cK\}, \quad \l\in\bC_+\cup\bC_.
\end{equation}
Formula \eqref{1.9} gives a bijective correspondence between all functions $\tau(\cd)\in\Rh$ and all equivalence classes of holomorphic operator pairs $(C_0(\l)\;\;C_1(\l)):\cH\oplus\cH\to\cK$ obeying
\begin{align*}
Im\,\l\cd Im(C_1(\l)C_0^*(\l))\geq 0, \qquad 0\in\rho (C_0(\l)-i\,\text{sgn}(Im\l) C_1(\l)),\\
C_1(\l)C_0^*(\ov\l)-C_0(\l)C_1^*(\ov\l)=0, \quad \l\in\bC_+\cup\bC_.\qquad\qquad
\end{align*}
(these relations follows from \eqref{1.7e}--\eqref{1.7g}). Similarly formulas \eqref{1.6} and \eqref{1.7a}--\eqref{1.7d} imply that the equality
\begin{equation}\label{1.9c}
\tau(\l)=\{K_0(\l),K_1(\l);\cK'\},\quad \l\in\bC_+\cup\bC_-
\end{equation}
gives a bijective correspondence between all functions $\tau(\cd)\in \Rh$ and all equivalence classes of holomorphic operator pairs $(K_0(\l)\;\;K_1(\l))^\top:\cK'\to\cH\oplus\cH$ obeying
\begin{align*}
Im\,\l\cd Im(K_0^*(\l)K_1(\l))\geq 0, \qquad 0\in\rho (K_1(\l)+i\,\text{sgn}(Im\l) K_0(\l)),\\
K_1^*(\ov\l)K_0(\l)-K_0^*(\ov\l)K_1(\l)=0, \quad \l\in\bC_+\cup\bC_-.\qquad\qquad
\end{align*}
Moreover in view of \eqref{1.8} and \eqref{1.8a} the connection between representations \eqref{1.9} and \eqref{1.9c} of the same function $\tau(\cd)\in \Rh$ is given by
\begin{equation*}
C_0(\l)=K_1^*(\ov\l), \quad C_1(\l)=-K_0^*(\ov\l),\quad \l\in\bC_+\cup\bC_-.
\end{equation*}
In the similar way one can reformulate the conditions (a)--(e) of Definition \ref{def1.0.4} for the class $\wt R^0(\cH):=\wt R^0(\cH,\cH)$.
In particular, a function $\tau (\cd):\bC_+\cup\bC_-\to \CG$ belongs to the class $\wt R^0(\cH)$ if and only if $\tau(\l)\equiv\t=\t^*(\in \CG), \;\l\in\bC_+\cup\bC_-$.
\end{remark}
\subsection{Boundary triplets and Weyl functions}
Let $A$ be a closed densely defined symmetric  operator in $\gH$. In what follows we will use the following notations:

$n_\pm (A):=\dim \gN_\l(A) \;\; (\l\in\bC_\pm)$ are deficiency indices of $A$;

$\exa$ is the set of all proper extensions of $A$, i.e., the set of all closed operators $\wt A$ in $\gH$ such that $A\subset\wt A\subset A^*$.

 Let $\cH_0$ be  a Hilbert space, let $\cH_1$ be a subspace
in $\cH_0$ and  let $\cH_2:=\cH_0\ominus\cH_1$. Denote by $P_j$  the
orthoprojector in $\cH_0$ onto $\cH_j,\; j\in\{1,2\}$.
\begin{definition}\cite{Mog06.2}\label{def1.1}
A collection $\Pi=\bta$, where $\G_j$ are linear mappings from $\cD (A^*)$ to
$\cH_j\; (j\in\{0,1\})$, is called a $D$-boundary triplet (or briefly
$D$-triplet) for $A^*$, if $\G=(\G_0\;\;\G_1)^\top :\cD (A^*)\to
\cH_0\oplus\cH_1 $ is a surjective linear mapping  onto $\cH_0\oplus\cH_1$ and
the following Green's identity holds
\begin {equation}\label{1.10}
(A^*f,g)-(f,A^*g)=(\G_1 f,\G_0 g)- (\G_0 f,\G_1 g)+i(P_2\G_0 f,P_2\G_0 g),
\quad f,g \in \cD (A^*).
\end{equation}
\end{definition}
As was shown in \cite{Mog06.2} a $D$-triplet $\bta$ for $A^*$ obeys the relation $\dim \cH_1=n_-(A)\leq n_+(A)=\dim \cH_0$. Moreover the equalities
\begin {equation}\label{1.11}
\cD (A_0):=\Ker \G_0=\{f\in\cD (A^*):\G_0 f=0\}, \qquad A_0=A^* \up \cD (A_0)
\end{equation}
define a maximal symmetric extension $A_0\in\exa$ with $n_-(A_0)=0$.

It turns out that for every $\l\in\bC_+\;(z\in\bC_-)$ the map $\G_0\up \gN_\l (A)\;(P_1\G_0\up \gN_z (A))$
 is an isomorphism. This makes it possible to introduce the operator functions ($\g$-fields) $\g_+(\cdot):\Bbb
C_+\to[\cH_0,\gH], \; \; \g_-(\cdot):\Bbb C_-\to[\cH_1,\gH]$ and the Weyl functions $M_+(\cdot):\bC_+\to [\cH_0,\cH_1], \;\; M_-(\cdot):\bC_-\to [\cH_1,\cH_0]$ by
\begin{align}
\g_+ (\l)=(\G_0\up\gN_\l (A))^{-1}, \;\;\;\l\in\Bbb C_+;\qquad
\g_- (z)=(P_1\G_0\up\gN_z (A))^{-1}, \;\;\; z\in\Bbb C_-,\label{1.13}\\
\G_1 \up \gN_\l (A)=M_+(\l)\G_0 \up \gN_\l (A),\quad \l\in\bC_+,\label{1.14}\qquad\qquad\qquad\\
(\G_1+iP_2\G_0)\up \gN_z (A)=M_-(z)P_1\G_0 \up \gN_z (A),\quad z\in
\bC_-.\qquad\qquad\label{1.15}
\end{align}
According to \cite{Mog06.2} all functions $\g_\pm$ and $M_\pm $ are holomorphic
on their domains and $M_+^*(\l)=M_-(\ov \l), \;\l\in\bC_+$. Moreover the following relation holds
\begin {equation}\label{1.15.0}
M_+(\mu)-M_+^*(\l)P_1+iP_2=(\mu-\ov\l)\g_+^*(\l)\g_+(\mu), \qquad \mu,\l\in\bC_+.
\end{equation}
\begin{lemma}\label{lem1.2}
Assume that $\bta$ is a $D$-triplet for $A^*$, $\t=\op\in\CA$ is an operator pair  \eqref{1.3}, $\t^\tm=\{(C_{0\tm},C_{1\tm}); \cK_\tm\}\in\CA$ is the $\tm$-adjoint pair and $C_{1*}, C_{0*}$ are operators \eqref{1.5}. Then
\begin{align}
(C_0\G_0+C_1\G_1)\g_+(\l)=C_0+C_1 M_+(\l), \;\;\;\l\in\bC_+ \label{1.15.1}\\
(C_{0\tm}\G_0+C_{1\tm}\G_1)\g_-(z)=C_{1*}+C_{0*} M_-(z), \;\;\;z\in\bC_-.
\end{align}
\end{lemma}
\begin{proof}
It follows from \eqref{1.13} -- \eqref{1.15} that
\begin{align}
\G_0\g_+(\l)=I_{\cH_0}, \quad \G_1\g_+(\l)=M_+(\l), \quad \l\in\bC_+\qquad\qquad\qquad\qquad\qquad \label{1.15.3}\\
P_1\G_0\g_-(z)=I_{\cH_1}, \quad P_2\G_0\g_-(z)=-iP_2M_-(z), \quad \G_1\g_-(z)=P_1M_-(z), \quad z\in\bC_-.\label{1.15.4}
\end{align}
The equality \eqref{1.15.1} is immediate from \eqref{1.15.3}. Moreover \eqref{1.15.4} yields
\begin{align*}
(C_{0\tm}\G_0+C_{1\tm}\G_1)\g_-(z)=C_{0\tm}P_1\G_0\g_-(z)+C_{0\tm}P_2 \G_0\g_-(z)+C_{1\tm}\G_1\g_-(z)=\\
C_{0\tm}\up\cH_1+(C_{1\tm}P_1-i C_{0\tm}P_2)M_-(z)= C_{1*}+C_{0*} M_-(z), \quad z\in\bC_-.
\end{align*}
\end{proof}
The combination of Proposition 4.1 and Theorem 4.2 from \cite{Mog06.2} gives the following theorem.
\begin{theorem}\label{th1.3}
Suppose that  $\Pi=\bta$ is a $D$-triplet for $A^*$, $\t=$ $\op$ $\in\CA$ is an operator pair and $\wt A\in\exa$ is an extension defined by the abstract boundary condition
\begin{equation}\label{1.16}
\cD (\wt A)=\{f\in\cD (A^*):\,C_0\G_0 f +C_1\G_1 f=0\}, \quad \wt A=A^*\up \cD (\wt A)
\end{equation}
Then $\l\in\rho (\wt A)\cap \bC_+\iff 0\in\rho (C_0+C_1 M_+(\l))$ and the following Krein type formula for  canonical resolvents holds
\begin{equation}\label{1.17}
(\wt A-\l)^{-1}=( A_0-\l)^{-1}-\g_+(\l)(C_0+C_1 M_+(\l))^{-1}C_1 \g_-^*(\ov\l), \quad \l\in \rho (\wt A)\cap \bC_+.
\end{equation}
\end{theorem}
Next recall the following definition.
\begin{definition}\label{def1.3a}
An operator function $\R:\bC_+\cup\bC_-\to\gH$ is called a generalized
resolvent of a symmetric operator  $A$, if there exist a
Hilbert space $\wt \gH\supset\gH$ and a self-adjoint operator $\wt A$ in $\wt\gH$ such that $A\subset \wt A$ and $\R =P_\gH (\wt A- \l)^{-1}\vert \gH,
\;\; \l \in \bC_+\cup\bC_-$.
\end{definition}
In the following theorem the description of all generalized resolvents is given in terms of abstract boundary conditions.
\begin{theorem}\label{th1.4}\cite{Mog06.2}
Let $\Pi=\bta$ be $D$-triplet for $A^*$. Then the relations
\begin{align}
\cD (\wt A(\l))= \left \{ \begin{array} {l}
\{f\in \cD (A^*):C_0(\l)\G_0 f-C_1(\l)\G_1 f=0\}, \;\; \l\in\bC_+\\
\{f\in \cD (A^*):D_0(\l)\G_0 f-D_1(\l)\G_1 f=0\}, \;\; \l\in\bC_- \end{array}\right. ,\;\;\;\; \wt A(\l)\in\exa \label{1.18}\\
\R=(\wt A(\l)-\l)^{-1}, \quad \l\in \bC_+\cup\bC_-\qquad\qquad\qquad\qquad \label{1.19}
\end{align}
establish a bijective correspondence between all  generalized resolvents $\R$ of $A$ and all collections of holomorphic operator pairs $\pair\in\RH$ defined by \eqref{1.7}. Moreover $\R$  is a canonical
resolvent if and only if $\tau  \in \RZ $.
\end{theorem}
\begin{remark}\label{rem1.5}
If a $D$-triplet $\Pi=\bta$ satisfies the relation $\cH_0=\cH_1:=\cH\;(\Leftrightarrow A_0=A_0^*)$, then it is a boundary triplet. More precisely this means that the collection $\Pi=\{\cH,\G_0,\G_1\}$ is a boundary triplet (boundary value space) for $A^*$ in the sense of \cite{GorGor}. In this  case the relations
\begin{equation}\label{1.19a}
\g(\l)=(\G_0\up\gN_\l(A))^{-1}, \qquad \G_1\up\gN_\l(A)=M(\l)\G_0 \up\gN_\l(A), \qquad \l\in\rho (A_0)
\end{equation}
define the operator function ($\g$-field) $\g(\cd):\rho (A_0)\to [\cH,\gH]$ and the Weyl function $M(\cd):\rho (A_0)\to [\cH]$ \cite{DM91} associated with operator functions  \eqref{1.13}--\eqref{1.15} via $\g(\l)=\g_\pm(\l)$ and $M(\l)=M_\pm(\l), \;\l\in\bC_\pm $.
Observe also that in the case $n_+(A)=n_-(A)$ formulas similar to \eqref{1.18},\eqref{1.19} were originally obtained in  \cite{Bru76} on the basis of Shtraus formula for resolvents \cite{Sht70} (see also \cite{DM91}).
\end{remark}
\section{Differential operators and decomposing boundary triplets}
\subsection{ Differential  operators }
Let $\D=[0,b\rangle\; (b\leq \infty)$ be an interval on the real axis (in the case $b<\infty$ the point $b$ may or may not belong to $\Delta$), let $H$ be a separable Hilbert space  with $\dim H\leq\infty$ and let
\begin{equation}\label{1.20}
l[y]=l_H[y]=\sum_{k=1}^n (-1)^k ( (p_{n-k}y^{(k)})^{(k)}-\tfrac {i}{2}
[(q_{n-k}^*y^{(k)})^{(k-1)}+(q_{n-k} y^{(k-1)})^{(k)}])+p_n y,
\end{equation}
be a differential expression of an even order $2n$ with smooth enough operator-valued coefficients $p_k(\cd), q_k(\cd):\D\to [H]$  such that $ p_k(t)=p_k^*(t)$ and $ 0\in\rho (p_0(t))$ for all $ t\in\D$ and $ k= 0\div n$.
Denote by $y^{[k]}(\cd), \; k=0\div 2n$ the quasi-derivatives of a vector-function $y(\cd):\D\to H$, corresponding to the expression \eqref{1.20}. Moreover for every operator function $Y(\cd):\D\to [\cK,H]$ ($\cK$ is a Hilbert space) introduce quasi-derivatives $Y^{[k]}(\cd)$  by the same formulas as $y^{[k]}$ (see \cite{Nai, Rof69}).

Let $\cD (l)$ be the set of all functions $y(\cd)$ such that  $y^{[k]}(\cd),\; k=0\div (2n-2)$ has a continuous derivative on $\D$ and $y^{[2n-1]}$ is absolutely continuous on $\D$. Furthermore for a given Hilbert space $\cK$ denote by $\cD_\cK (l)$ the set of all operator-functions $Y(\cd)$ with values in $[\cK,H]$ such that $Y^{[k]}(\cd), \;k=0\div (2n-1) $ has a continuous derivative on $\D$. Clearly  for every $y\in \cD (l)$ and $Y\in \cD_\cK (l)$ the functions $y^{[k]}(\cd):\D\to H,\; k=0\div (2n-1) $ and $Y^{[k]}(\cd):\D\to [\cK,H],\;k=0\div 2n $ are continuous on $\D$, the function $y^{[2n]}(t)(\in H)$ is defined almost everywhere  on $\D$ and
\begin{equation*}
l[y]=y^{[2n]}(t), \quad y\in\cD (l); \qquad  l[Y]=Y^{[2n]}(t), \quad Y\in\cD_\cK (l).
\end{equation*}
This makes it possible to introduce the vector functions $y^{(j)}(\cd):\D\to H^n, \; j\in\{1,2\}$ and $\wt y(\cd):\D\to H^n\oplus H^n$,
\begin{align}
y^{(1)}(t):=\{y^{[k-1]}(t)\}_{k=1}^n (\in H^n), \qquad y^{(2)}(t):=\{y^{[2n-k]}(t)\}_{k=1}^n (\in H^n),\label{1.21}\\
\wt y(t)=\{y^{(1)}(t),y^{(2)}(t)\} (\in H^n\oplus H^n), \qquad t\in\D,\qquad\qquad \label{1.22}
\end{align}
which correspond to every $y\in\cD (l)$. Similarly with each $Y\in \cD_\cK (l)$  we associate the operator-functions $Y^{(j)}(\cd):\D\to [\cK,H^n]$ and $\wt Y(\cd):\D\to [\cK,H^n\oplus H^n]$ given by
\begin{align*}
Y^{(1)}(t)=(Y(t)\;\;Y^{[1]}(t)\;\dots\; Y^{[n-1]}(t))^\top, \quad  Y^{(2)}(t)=(Y^{[2n-1]}(t)\;\;Y^{[2n-2]}(t)\;\dots \; Y^{[n]}(t))^\top,\\
\wt Y(t)=(Y^{(1)}(t)\;\;\;Y^{(2)}(t))^\top : \cK \to  H^n\oplus H^n,\qquad t\in\D.\qquad\qquad\qquad
\end{align*}

Next for a given $\l\in\bC$ consider the equation
\begin{equation}\label{1.23}
l[y]-\l y=0.
\end{equation}
As is known this equation has the unique vector solution $y\in\cD (l)$ (operator solution $Y\in \cD_\cK (l)$) with the given initial data $y_{j0}= y^{(j)}(0)$ (respectively, $ Y_{j0}= Y^{(j)}(0)$), $j\in\{1,2\}$. We distinguish the two "canonical" operator solutions $c (\cd,\l)$ and $s (\cd,\l):\D\to [H^n,H], \; \l\in\bC$ of the equation \eqref{1.23} with the initial data
\begin{equation}\label{1.24}
c^{(1)}(0,\l)=I_{H^n}, \quad c ^{(2)}(0,\l)=0, \quad s ^{(1)}(0,\l)=0,
\quad s ^{(2)}(0,\l)=I_{H^n}, \quad \l\in\bC.
\end{equation}
Clearly the equality
\begin{equation}\label{1.25}
Y_0(t,\l):= (c(t,\l)\;\; s(t,\l)): H^n\oplus H^n\to H, \quad \l\in\bC
\end{equation}
defines the operator solution of \eqref{1.23} with $\wt Y_0 (0,\l)=I_{H^n\oplus H^n}$. Moreover every operator solution $Y(\cd)\in \cD_\cK (l)$ of \eqref{1.23} is connected with $Y_0(\cd, \l)$ by
\begin{equation}\label{1.25a}
Y(t)=Y_0(t,\l)\wt Y(0).
\end{equation}

The following lemma is well known (see for instance \cite{Nai,Hol85}).
\begin{lemma}\label{lem1.6}
Let $f\in L_{1,loc}(\D; H)$ and let $Y(\cd):\D\to [\cK,H]$ be an operator solution of \eqref{1.23}. Next assume that an absolutely continuous   function $C(\cd):\D\to\cK$ satisfies
\begin{equation}\label{1.26}
\wt Y(x) C'(x)=\hat f(x) \quad (\text
{mod} \;\; \mu),
\end{equation}
where the vector function $\hat f(\cd):\D\to H^n\oplus H^n$ is given by $\hat f(x)=\{\underbrace{0,\dots, 0,}_n  -f(x),\underbrace {0,\dots, 0}_{n-1}\}$ and $\mu$ is the Lebesgue measure on $\D$. Then the vector function
y(x):=Y(x)C(x) belongs to $\cD (l)$ and obeys the relations
\begin{equation}\label{1.27}
l[y]-\l y=f, \qquad \wt y (x)=\wt Y(x) C(x).
\end{equation}
\end{lemma}
In what follows we  denote by $\gH(=L_2(\D; H))$ the Hilbert space of all measurable  functions $f(\cd):\D\to H$ such that $\int _0^b ||f(t)||^2\,dt<\infty$. Moreover  $\LH{\cK}$ stands for the set of all operator-functions $Y(\cd):\D\to [\cK, H]$ such that $Y(t)h\in \gH$
for all $h\in\cK$.

It  is known \cite{Nai,Rof69} that the expression \eqref{1.20} generate the maximal operator $L$ in $\gH$, defined on the domain
$\cD=\cD (L):=\{y\in\cD (l)\cap \gH: l[y]\in\gH\}$
by the equality $Ly=l[y], \; y\in\cD$. Moreover  the Lagrange's identity
\begin{equation}\label{1.28}
(Ly,z)_\gH -(y,Lz)_\gH=[y,z](b)-[y,z](0),\qquad y,z\in\cD
\end{equation}
holds with
\begin{equation*}
[y,z](t)=(y^{(1)}(t),z^{(2)}(t))_{H^n}-(y^{(2)}(t),z^{(1)}(t))_{H^n}, \quad [y,z](b)=\lim_{t\uparrow b} [y,z](t).
\end{equation*}
The minimal operator $L_0$  is defined as a restriction of $L$ onto the domain $D_0=\cD(L_0)$ of all functions $y\in\cD$ such that $\wt y(0)=0$ and $[y,z](b)=0$ for all $z\in\cD$. As is known \cite {Nai,Rof69}  $L_0$ is a closed densely defined symmetric operator  in $\gH$ and $L_0^*=L$. Moreover the  subspace $\gN_\l(L_0)\,(=\Ker (L-\l))$ is the set of all solutions of \eqref{1.23} belonging to $\gH$.

Let $\t=\t^*\in\C (H^n)$  and let $L_\t$ be a  symmetric extension of $L_0$ with the domain
\begin{equation}\label{1.29}
\cD(L_\t)=\{y\in\cD:\wt y(0)\in\t\;\;\text{and}\;\;[y,z](b)=0,\;z\in\cD\}.
\end{equation}
As was shown in  \cite{Mog} deficiency indices $n_\pm(L_\t)$ of an operator $L_\t$ do not depend on $\t(=\t^*)$. This makes it possible to introduce the following definition.
\begin{definition}\label{def1.6a}\cite{Mog}
The numbers $n_{b\pm}:=n_\pm(L_\t)\;(\t=\t^*\in\C (H^n))$ are called deficiency indices of the differential  expression $l[y]$ at the right end $b$  of the interval $\D$.
\end{definition}
\subsection{Decomposing boundary triplets}
Assume that $\cH'_1$ is a subspace in a Hilbert space $\cH'_0$, $\cH_2':=\cH'_0\ominus\cH_1'$,  $\G_0':\cD\to \cH'_0$ and $\G_1':\cD\to \cH'_1$ are linear maps and  $P'_j$ is the orthoprojector in $\cH'_0$ onto $\cH'_j, \; j\in\{1,2\}$. Moreover let $\cH_0=H^n\oplus \cH'_0, \;\cH_1=H^n\oplus \cH'_1$ and let $\G_j:\cD\to \cH_j,\;j\in\{0,1\} $ be linear maps given by
\begin{equation}\label{1.30}
\G_0 y=\{y^{(2)}(0),  \G'_0y\}\,(\in H^n\oplus \cH'_0), \;\;\; \G_1 y=\{-y^{(1)}(0),  \G'_1y\}\,(\in H^n\oplus \cH'_1), \quad y\in\cD.
\end{equation}
\begin{definition}\label{def1.9b}\cite{Mog}
A collection $\Pi=\bta$, where $\G_0$ and $\G_1$ are linear maps \eqref{1.30}, is said to be a decomposing $D$-triplet for $L$ if the map $\G'=(\G'_0\;\;\G'_1)^\top:\cD\to\cH_0'\oplus\cH'_1$ is surjective and the following identity holds
\begin{equation}\label{1.31}
[y,z](b)=(\G'_1 y,\G'_0 z)- (\G'_0 y,\G'_1 z)+i(P'_2\G'_0 y,P'_2\G'_0 z),
\quad y,z \in \cD.
\end{equation}
\end{definition}
In the case $\cH'_0=\cH'_1=:\cH'\;(\iff \cH_0=\cH_1=:\cH)$ a decomposing $D$-triplet $\Pi=\bt$ is called a decomposing boundary triplet for $L$. For such a triplet the identity \eqref{1.31} takes the form
\begin{equation}\label{1.32}
[y,z](b)=(\G'_1 y,\G'_0 z)- (\G'_0 y,\G'_1 z), \quad y,z \in \cD.
\end{equation}
As was shown in \cite{Mog}, Lemma 3.4 a decomposing $D$-triplet (a decomposing boundary triplet) for $L$ is a $D$-triplet (a boundary triplet) in the sense of Definition \ref{def1.1} and Remark \ref{rem1.5}. Moreover a decomposing $D$-triplet (boundary triplet) for $L$  exists if and only if $n_{b-}\leq n_{b+}$ (respectively, $n_{b-}= n_{b+}$), in which case
\begin{equation}\label{1.32.0}
\dim\cH_1'=n_{b-}\leq n_{b+}=\dim\cH_0'\quad (\text{respectively,}\;\; n_{b-}=n_{b+}=\dim\cH').
\end{equation}
Therefore in the sequel  we suppose (without loss of generality)   that $n_{b-}\leq n_{b+}$ and, consequently, $n_-(L_0)\leq n_+(L_0)$.
\begin{theorem}\label{th1.10}\cite{Mog}
Let $\Pi=\bta$ be a decomposing $D$-triplet \eqref{1.30}   and let
\begin{align}
M_+(\l)=\begin{pmatrix} m(\l) & M_{2+}(\l)\cr M_{3+}(\l) & M_{4+}(\l)\end{pmatrix}:
H^n\oplus\cH'_0 \to H^n\oplus \cH'_1, \quad \l\in\bC_+,\label{1.32a}\\
M_-(z)=\begin{pmatrix} m(z) & M_{2-}(z)\cr M_{3-}(z) & M_{4-}(z)\end{pmatrix}:
H^n\oplus\cH'_1 \to H^n\oplus \cH'_0, \quad z\in\bC_-,\label{1.32b}
\end{align}
be the block-matrix representations of the  Weyl functions \eqref{1.14}, \eqref{1.15} for $\Pi$. Then:

1) the maximal symmetric extension $A_0\in\exl$ (see \eqref{1.11}) has the domain
\begin{equation}\label{1.33}
\cD (A_0)=\{y\in\cD : y^{(2)}(0)=0,\; \G'_0 y=0\};
\end{equation}
2) for every $\l\in\bC_+\cup\bC_-$ there exists the unique operator function $v_0(\cd,\l)\in\LH{H^n}$, satisfying the equation  \eqref{1.23}and the boundary conditions
\begin{align}
 v_0^{(2)}(0,\l)=I_{H^n}, \quad \l\in\bC_+\cup\bC_-\qquad\qquad\qquad \label{1.34}\\
\G'_0(v_0(t,\l)\hat h)=0,\;\;\;\l\in\bC_+ ;\qquad P'_1\G'_0(v_0(t,z)\hat h)=0,\;\;\;z\in\bC_-, \;\;\;\; \hat h\in H^n.\label{1.35}
\end{align}
Moreover for every $\l\in\bC_+\;(z\in\bC_-)$ there exists the  unique operator function $u_+(\cd,\l)\in\LH{\cH'_0}\;(u_-(\cd,z)\in\LH{\cH'_1})$, satisfying   \eqref{1.23} and the boundary conditions
\begin{align}
 u_+^{(2)}(0,\l)=0, \quad \G'_0 (u_+(t,\l) h'_0)=h_0',\quad \l\in\bC_+, \;\;\;h'_0\in\cH'_0;\;\;\;\; \label{1.36}\\
 u_-^{(2)}(0,z)=0, \quad P_1'\G'_0(u_-(t,z) h'_1)=h_1',\quad z\in\bC_-,\;\;\; h'_1\in\cH'_1. \label{1.37}
\end{align}

3) let $Z_+(\cd,\l)\in\LH{\cH_0}$ and $Z_-(\cd,z)\in\LH{\cH_1}$ be operator solutions of \eqref{1.23} defined by the block-matrix representations
\begin{align}
Z_+(t,\l)=(v_0(t,\l)\;\;u_+(t,\l)):H^n\oplus\cH'_0\to H, \quad \l \in \bC_+, \label{1.38}\\
Z_-(t,z)=(v_0(t,z)\;\;u_-(t,z)):H^n\oplus\cH'_1\to  H, \quad z\in\bC_-\label{1.39}.
\end{align}
Then
\begin{align}
\wt Z_+(0,\l)=\begin{pmatrix}v_0^{(1)}(0,\l) &  u_+^{(1)}(0,\l) \cr v_0^{(2)}(0,\l) &  u_+^{(2)}(0,\l)\end{pmatrix}=
\begin{pmatrix}-m(\l) & -M_{2+}(\l)\cr I_{H^n} & 0\end{pmatrix}, \;\;\l\in\bC_+\label{1.37a}\\
\wt Z_-(0,z)=\begin{pmatrix}v_0^{(1)}(0,z) &  u_-^{(1)}(0,z) \cr v_0^{(2)}(0,z) &  u_-^{(2)}(0,z)\end{pmatrix}=\begin{pmatrix}-m(z) & -M_{2-}(z)\cr I_{H^n} & 0\end{pmatrix}, \;\; z\in\bC_-.\label{1.37b}
\end{align}
and the $\g$-fields \eqref{1.13} for $\Pi$ obey the relations
\begin{align}
(\g_+(\l) h_0)(t)=Z_+(t,\l) h_0, \quad \l\in\bC_+,\;\;  h_0\in \cH_0 \label{1.40}\\
(\g_-(z) h_1)(t)=Z_-(t,z) h_1, \quad z\in\bC_-,\;\;  h_1\in \cH_1 \label{1.41}
\end{align}
This implies that $Z_+(\cd,\l)$ and $Z_-(\cd,\l)$ are fundamental solutions of \eqref{1.23} with holomorphic quasi-derivatives $Z_+^{[k]}(t,\cd)$ and $Z_-^{[k]}(t,\cd), \; k=0\div (2n-1)$ (see Definition \ref{def1.7}).

4) the block-matrix representations \eqref{1.32a} and \eqref{1.32b} generate the uniformly strict Nevanlinna  operator-function $m(\cd):\bC_+\cup\bC_-\to [H^n]$ ( this means that $Im \l\cd Im\; (m(\l))\geq 0, \; 0\in\rho (Im\; (m(\l)))$
and $m^*(\ov\l)=m(\l), \; \l\in\bC_+\cup\bC_-$).
\end{theorem}
\begin{definition}\label{def1.10a}\cite{Mog}
The operator function $m(\cd)$ defined in the statement 4) of Theorem \ref{th1.10} is called an $m$-function of the operator $L_0$, corresponding to the extension $A_0$.
\end{definition}
In the following corollary the statements of Theorem \ref{th1.10} are reformulated for the case of a decomposing boundary triplet.
\begin{corollary}\label{cor1.11}
Let $\Pi=\bt$ be a decomposing boundary triplet \eqref{1.30} for $L$ (that is $\cH_0'=\cH_1'=:\cH'$), let $A_0(=L\up\Ker\G_0)$ be the extension \eqref{1.11}  and let
\begin{align}\label{1.42}
M(\l)=\begin{pmatrix} m(\l) & M_{2}(\l)\cr M_{3}(\l) & M_{4}(\l)\end{pmatrix}:
H^n\oplus\cH' \to H^n\oplus \cH', \quad \l\in\rho (A_0)
\end{align}
be the block-matrix representation of the  corresponding Weyl function \eqref{1.19a}. Then:

1) $A_0$ is a selfadjoint extension with the domain \eqref{1.33};

2) for every $\l\in\rho (A_0)$ there exist the unique pair of operator functions $v_0(\cd,\l)\in\LH{H^n}$ and $u_0(\cd,\l)\in\LH{\cH'}$
satisfying \eqref{1.23}and the boundary conditions
\begin{align*}
v_0^{(2)}(0,\l)=I_{H^n},\quad \G'_0(v_0(t,\l)\hat h)=0,\;\;\;\hat h\in H^n, \;\;\l\in\rho (A_0)\\
u_0^{(2)}(0,\l)=0, \quad  \G'_0 (u_0(t,\l) h')=h', \;\;\; h'\in\cH',\;\;\l\in\rho (A_0)
\end{align*}

3) let $Z_0(\cd,\l)\in\LH{\cH}$  be an operator solution of \eqref{1.23} defined by
\begin{align}\label{1.43}
Z_0(t,\l)=(v_0(t,\l)\;\;u_0(t,\l)):H^n\oplus\cH'\to H, \quad \l \in \rho (A_0)
\end{align}
Then
\begin{align*}
\wt Z_0(0,\l)=\begin{pmatrix}v_0^{(1)}(0,\l) &  u_0^{(1)}(0,\l) \cr v_0^{(2)}(0,\l) &  u_0^{(2)}(0,\l)\end{pmatrix}=
\begin{pmatrix}-m(\l) & -M_{2}(\l)\cr I_{H^n} & 0\end{pmatrix}, \;\;\l\in\rho (A_0)
\end{align*}
and the corresponding $\g$-field \eqref{1.19a}  obey $(\g(\l)h)(t)=Z_0(t,\l)h,\;h\in\cH,\;\l\in\rho (A_0)$.
\end{corollary}
\section{Fundamental solutions and resolvents of differential operators}
\subsection{Fundamental solutions and spectra of proper extensions}
Let $\Pi=\bta$ be a decomposing $D$-triplet  \eqref{1.30} for $L$, let $\t=\op$ $\in\CA$ be an   operator pair and let $\t^\times =\{(C_{0\times},C_{1\times});\cK_\times\}\in\CA$ be a $\times$-adjoint operator pair.
Moreover assume that
\begin{align}
C_0=(\hat C_2\;\;C_0'):H^n\oplus\cH'_0\to \cK,\quad  C_1=(-\hat C_1\;\;C_1'):H^n\oplus\cH'_1\to \cK,\qquad\quad\label{2.3}\\
C_{0\times}=(\hat C_{2\times}\;\;C_{0\times}'):H^n\oplus\cH'_0\to \cK_\times,\quad  C_{1\times}=(-\hat C_{1\times}\;\;C_{1\times}'):H^n\oplus\cH'_1\to \cK_\times \label{2.4}
\end{align}
are the block-matrix representations of the operators $C_j$ and $C_{j\times}, \; j\in\{0,1\}$.

In the next theorem the set of all proper extensions of the minimal operator $L_0$ is described in terms of boundary conditions.
\begin{theorem}\label{th2.1}\cite{Mog}
Let $\Pi=\bta$ be a decomposing $D$-triplet \eqref{1.30} for $L$. Then the equalities (the boundary conditions)
\begin{equation}\label{2.5}
\cD (\wt A)=\{y\in\cD: \,\hat C_1 y^{(1)}(0)+ \hat C_2 y^{(2)}(0)+C'_0\G'_0 y+C'_1\G'_1 y=0\}, \quad \wt A=L\up \cD (\wt A)
\end{equation}
establish a bijective correspondence between all proper extensions $\wt A\in\exl$ and all operator pairs $\t=\op\in\CA$ defined by  \eqref{2.3}. Moreover the adjoint $\wt A^*$ to the extension \eqref{2.5} has the domain
\begin{equation}\label{2.6}
\cD (\wt A^*)=\{y\in\cD: \,\hat C_{1\times} y^{(1)}(0)+ \hat C_{2\times} y^{(2)}(0)+C'_{0\times}\G'_0 y+C'_{1\times}\G'_1 y=0\}
\end{equation}
where the operators $\hat C_{1\times}, \; \hat C_{2\times}, \; C'_{0\times}\ $ and  $C'_{1\times}$ are defined by \eqref{2.4}
\end{theorem}
\begin{remark}\label{rem2.2}
It is clear that the boundary conditions \eqref{2.5} can be written as
\begin{equation}\label{2.8a}
\cD (\wt A)=\{y\in\cD:\; C_0\G_0 y + C_1\G_1 y=0\}.
\end{equation}
\end{remark}
\begin{definition}\label{def1.7}
Let $\l\in\bC$ and let $\cK'$ be a Hilbert space.
An operator function
\begin{equation}\label{29.0}
Z(\cd, \l):\D\to [\cK',H]
\end{equation}
will be called a fundamental solution of the equation \eqref{1.23} if: (i) $Z(\cd, \l)$ is an operator solution of \eqref{1.23}; (ii) $Z(\cd, \l)\in \LH {\cK'}$ (and, hence, $Z(t,\l)h'\in\gN_\l(L_0)$ for all $h'\in\cK'$); (iii) for every $y\in\gN_\l(L_0)$ there is the unique $h'\in\cK'$ such that  $y(t)=Z(t,\l)h'$.
\end{definition}
Clearly, an operator function \eqref{29.0} is a fundamental solution of the equation \eqref{1.23} if and only if it is an operator solution of \eqref{1.23} and the equality $y=Z(t,\l)h'$ gives a bijective correspondence between all functions $y\in\gN_\l(L_0)$ and all $h'\in\cK'$.
\begin{lemma}\label{lem1.8}
1) For every operator  $Z\in [\cK', \gN_\l(L_0)]$ the relation
\begin{equation}\label{29.1}
Z(t,\l)h'=(Zh')(t),\quad h'\in\cK'
\end{equation}
define an operator solution $Z(\cd, \l)\in L_2' [\cK',H]$ of \eqref{1.23}. Conversely for each such a solution  there exists an operator $Z\in [\cK', \gN_\l(L_0)]$ such that \eqref{29.1} holds.

2) The equality \eqref{29.1} establishes a bijective correspondence between all fundamental solutions \eqref{29.0} of the equality \eqref{1.23} and all isomorphisms $Z\in [\cK', \gN_\l(L_0)]$.

3) Let $Z(\cd, \l)\in \LH{\cK'}$ be an operator solution of \eqref{1.23} and let $Z\in [\cK',\gN_\l(L_0)]$ be the corresponding operator \eqref{29.1} considered  as acting to the Hilbert space $\gH$. Then
\begin{equation}\label{29.2}
Z^*f=\smallint_0^b Z^*(t,\l)f(t)\, dt:=\lim_{\eta\uparrow b}\smallint_0^\eta Z^*(t,\l)f(t) \, dt, \quad f=f(t)\in\gH.
\end{equation}
\end{lemma}
\begin{proof}
1) Let $\d_\l:\gN_\l(L_0)\to H^n\oplus H^n$ be a linear map given by $\d_\l y=\wt y(0), \; y\in \gN_\l(L_0)$. As is known \cite{Rof69,HolRof} the operator $\d_\l$ is bounded, that is $\d_\l \in [\gN_\l(L_0), H^n\oplus H^n]$.

Next for a given $Z\in [\cK', \gN_\l(L_0)]$ consider the operator solution $Z(\cd,\l):\D\to [\cK',H]$ of the equation \eqref{1.23} with the initial data $\wt Z(0,\l)=\d_\l Z $. It is clear that for every fixed $h'\in\cK'$ the vector function
\begin{equation}\label{29.3}
y=y(t,h'):=(Zh')(t)
\end{equation}
is a solution of \eqref{1.23} with $\wt y(0)=\d_\l Zh'=\wt Z(0,\l)h'$. Therefore $y=Z(t,\l)h'$ and by \eqref{29.3} the introduced operator function $Z(\cd,\l)$ satisfies \eqref{29.1}. Hence $Z(\cd,\l)\in\LH{\cK'}$.

Conversely, let $Z(\cd,\l)\in\LH{\cK'}$ be an operator solution of \eqref{1.23}. Then the equality \eqref{29.1} defines a linear map $Z:\cK'\to \gN_\l(L_0) $ obeying $\d_\l Z =\wt Z (0,\l)$  with bounded operators $\d_\l$ and $\wt Z(0,\l)$. This and the equality  $\Ker\,\d_\l=\{0\}$ imply that the operator $Z$ is closed and, consequently, bounded.

The statement 2) is immediate from 1).

3) The proof of   \eqref{29.2} is similar to that of the formula (3.70) in \cite{Mog}
\end{proof}
The following theorem is immediate from Lemma \ref{lem1.8}, 2).
\begin{theorem}\label{th1.9}
1) For every $\l\in\bC$ and  for every Hilbert space $\cK'$  with
\begin{equation}\label{29.4}
\dim \cK'=\dim\gN_\l (L_0)
\end{equation}
there exists a fundamental solution $Z(\cd,\l):\D\to [\cK',H]$ of the equation \eqref{1.23}. Conversely, for  every fundamental solution  \eqref{29.0} the equality \eqref{29.4} holds.

2) Let $Z_0(\cd,\l):\D\to [\cK',H]$ be a fundamental solution of
\eqref{1.23}. Then the equality
\begin{equation*}
Z(t,\l)=Z_0(t,\l)X
\end{equation*}
gives a bijective correspondence between all fundamental solutions
\eqref{29.0} and all bounded isomorphisms $X\in [\cK']$.
\end{theorem}
In the following theorem we show that there exist fundamental  solutions  $Z(t,\l)$ of \eqref{1.23} with holomorphic quasi-derivatives of all orders.
\begin{theorem}\label{th1.9a}
Let $\wt A$ be a proper extension of $L_0$ with nonempty resolvent set $\rho (\wt A)$. Then there exists a family of fundamental solutions $Z(\cd,\l):\D\to [\cK',H], \;\l\in\rho (\wt A)$ of the equation \eqref{1.23} such that for every fixed $t\in\D$ all quasi-derivatives  $Z^{[k]}(t,\cd), \;k=0\div (2n-1)$ are holomorphic on $\rho (\wt A)$.
\end{theorem}
\begin{proof}
Denote by $\cD_+$ the Hilbert space of all functions $y\in\cD$ with the inner product
\begin{equation*}
(y,z)_+=(y,z)_\gH +(Ly,Lz)_\gH, \qquad y,z\in\cD_+
\end{equation*}
and let $\d:\cD_+\to H^n\oplus H^n$ be a linear map given by $\d y=\wt y(0), \; y\in\cD_+$.  Since the norms $||\cd||_+$ and $||\cd||_\gH$ are equivalent on a subspace $\gN_\l (L_0)$, the operator $\d \up \gN_\l (L_0)(=\d_\l)$ is bounded. Moreover $\d\up \cD_0=0$ by definition of $\cD_0$. This and the orthogonal decomposition
\begin{equation*}
\cD_+=\cD_0\oplus \gN_i(L_0)\oplus \gN_{-i}(L_0)
\end{equation*}
imply that $\d\in [\cD_+,H^n\oplus H^n]$.

Next assume that $\l_0\in\rho (\wt A)$ and $\cZ_0$ is an isomorphism of a Hilbert space $\cK'$ onto $\gN_{\l_0}(L_0)$. Since the resolvent $(\wt A-\l)^{-1} \; (\l\in\rho(\wt A))$ is a holomorphic operator function with values in $[\gH,\cD_+]$, the equality
\begin{equation}\label{29.5}
\cZ(\l):= \cZ_0+(\l-\l_0)(\wt A-\l)^{-1}\cZ_0, \quad \l\in \rho (\wt A)
\end{equation}
defines a holomorphic operator function $\cZ(\cd):\rho (\wt A)\to [\cK', \cD_+]$. Moreover one can easily verify that $\cZ(\l)\cK'=\gN_\l(L_0)$ and $\Ker \cZ(\l)=\{0\}$. Therefore by Lemma\ref{lem1.8}, 2)  the relation
\begin{equation}\label{29.6}
\cZ(t,\l)h':=(\cZ(\l)h')(t), \quad h'\in\cK', \quad t\in\D
\end{equation}
 defines a family of fundamental solutions of the equation \eqref{1.23} with $\wt \cZ(0,\l)=\d \cZ(\l)$. Hence the operator function $\wt \cZ(0, \cd)$  is holomorphic on $\rho (\wt A)$ and, consequently, so is the function $\wt \cZ(t, \cd)$ for every fixed $t\in\D$.
\end{proof}

Let $\Pi=\bta$ be a decomposing $D$-triplet \eqref{1.30} for $L$ and let $Z(\cd, \l):\D\to [\cK',H]\; (\l\in\bC)$ be a fundamental solution of the equation \eqref{1.23}. For every operator pair $\t=\op$ $\in\CA$ introduce the operator $T\in [\cK',\cK]$ by
\begin{equation}\label{2.11}
T=(C_0\G_0 + C_1\G_1)Z,
\end{equation}
where $Z\in [\cK', \gN_\l(L_0)]$ is the corresponding  isomorphism \eqref{29.1} (see Lemma \ref{lem1.8}, 2)).   Using the block-matrix representation \eqref{2.3} one rewrites \eqref{2.11} as
\begin{equation}\label{2.12}
Th'=(\hat C_1 Z^{(1)}(0,\l)+ \hat C_2 Z^{(2)}(0,\l))h'+(C_0'\G_0' + C_1'\G_1')(Z(t,\l)h'),\quad h'\in\cK'.
\end{equation}
In the following theorem we describe the spectrum of a proper extension $\wt A\in\exl$ in terms of boundary conditions and a fundamental solution $Z(\cd,\l)$.
\begin{theorem}\label{th2.3}
Let under the above suppositions $\wt A\in\exl$ be a proper extension given by the boundary conditions \eqref{2.5}. Then for every $\l\in\hat\rho (L_0)$ the following holds
\begin{align}
\l\in\rho (\wt A)\iff 0\in \rho (T), \qquad \l\in \s_j (\wt A)\iff 0\in \s_j (T),
\;\; j\in\{p,c,r\} \quad\label{2.13}\\
\l\in\hat\rho (\wt A)\iff 0\in \hat\rho (T),\qquad  \ov {\cR (\wt  A-\l)}=\cR (\wt  A-\l) \iff \ov {\cR (T)}= \cR
(T)\label{2.14}
\end{align}
\end{theorem}
\begin{proof}
By using Proposition 3.17 in \cite{MalMog02}  one can derive the relations  \eqref{2.13} -- \eqref{2.15} with the operator $(C_0\G_0 + C_1\G_1)\up \gN_\l (L_0)$ in place of $T$. This and  bijectivity  of $Z$ in \eqref{2.11} yield the statement of the theorem.
\end{proof}
\begin{remark}\label{rem2.4}
In \cite{RofHol84,HolRof} a fundamental solution of the equation \eqref{1.23} is defined as an operator function \eqref{29.0} obeying the conditions (i)--(iii) of Definition \ref{def1.7} and, in addition, a selfadjoint boundary condition at the point $b$ (which exists only in the case $n_{b+}=n_{b-}$). In this connection note that statements of Theorems \ref{th1.9a} and \ref{th2.3} complement and generalize similar results obtained in \cite{RofHol84,HolRof}.
\end{remark}
\subsection{Resolvents of proper extensions of the minimal operator}
Let as before $\Pi=\bta$ be a decomposing $D$-triplet \eqref{1.30} for $L$, let $\t=\op$ $\in\CA$ be an operator pair \eqref{2.3} and let $\wt A\in\exl$ be the corresponding extension \eqref{2.5}. Assume that $\l\in \rho (\wt A)$ and let $Z(\cd, \l):\D\to [\cK',H] $ be a fundamental solution of the equation \eqref{1.23} (such a solution exists in view of Theorem \ref{th1.9}). It follows from \eqref{2.13} that the corresponding operator $T\in [\cK',\cK]$ (see \eqref{2.11}) is invertible. This allows us to introduce the operator function $Y_\t(\cd, \ov\l):\D\to [\cK',H]$ as the operator solution of the equation $l[y]-\ov\l y=0$ with the initial data
\begin{equation}\label{2.18}
Y_\t^{(1)}(0,\ov\l)=-\hat C_2^* T^{-1*}, \qquad Y_\t^{(2)}(0,\ov\l)=\hat C_1^* T^{-1*}.
\end{equation}
Next assume that  $\t^\tm =\{(C_{0\times},C_{1\times});\cK_\times\} \in\CA$ is a $\tm$-adjoint operator pair \eqref{2.4}. Then the adjoint extension $\wt A^*$ is defined by \eqref{2.6} and $\ov\l\in \rho (\wt A^*)$. Let $Z(\cd, \ov\l):\D\to [\cK_\tm',H]$ be a fundamental solution of the equation $l[y]-\ov\l y=0$ and let $T_\tm\in [\cK_\tm',\cK_\tm]$
be the operator given by
\begin{equation}\label{2.19}
T_\tm h'=(\hat C_{1\tm} Z^{(1)}(0,\ov\l)+ \hat C_{2\tm} Z^{(2)}(0,\ov\l))h'+ (C_{0\tm}'\G_0' + C_{1\tm}'\G_1')(Z(t,\ov\l)h'), \;\;\; h'\in\cK_\tm'.
\end{equation}
Then in view of Theorem \ref{th2.3} $0\in\rho (T_\tm)$, which makes it possible to define the operator function $Y_{\t^\tm}(\cd, \l):\D\to [\cK_\tm',H]$ as the operator solution of \eqref{1.23} with the initial data
\begin{equation}\label{2.20}
Y_{\t^\tm}^{(1)}(0,\l)=-\hat C_{2\tm}^* T_\tm^{-1*}, \qquad Y_{\t^\tm}^{(2)}(0,\l)=\hat C_{1\tm}^* T_\tm ^{-1*}.
\end{equation}
One can easily verify that for  a given fundamental solution $Z(\cd,\l)$ ($Z(\cd, \ov\l)$) the operator function $Y_{\t}(\cd, \ov\l)$ ($Y_{\t^\tm}(\cd, \l)$) is uniquely defined by $\t$, i.e., it does not depend on the choice of the equivalent representation $\t=\{(C_{0}, C_{1}); \cK\}$ (respectively,  $\t^\tm=\{(C_{0\times},C_{1\times});\cK_\times\} $).
\begin{definition}\label{def2.6}
The operator function $G_\t(\cd,\cd,\l):\D\tm\D\to [H]$, given by
\begin{equation}\label{2.22}
G_\t (x,t,\l)=\begin{cases} Z(x,\l)\, Y_\t^*(t,\ov\l), \;\; x>t \cr Y_{\t^\tm}(x,\l)\, Z^* (t,\ov\l), \;\; x<t \end{cases}, \quad \l\in\rho (\wt A)
\end{equation}
will be called the Green function, corresponding to an operator pair $\t\in\CA$.
\end{definition}
It is easily seen that for a given operator pair $\t$ the Green function \eqref{2.22} does not depend on the choice of fundamental solutions $Z(\cd,\l)$ and  $Z(\cd, \ov\l)$.
\begin{theorem}\label{th2.8}
Suppose that $\Pi=\bta$ is a decomposing $D$-triplet \eqref{1.30} for $L$, $\t=\op\in\CA$ is the operator pair \eqref{2.3}, $\wt A\in\exl$ is the corresponding extension \eqref{2.5}, $\l\in\rho (\wt A)$ and $G_\t (x,t,\l)$ is the Green function \eqref{2.22}. Then the resolvent  $(\wt A-\l)^{-1}\in [\gH]$ is the integral operator, given  by
\begin{equation}\label{2.24}
((\wt A-\l)^{-1}f)(x)=\smallint\limits_0^b  G_\t (x,t,\l)f(t)\, dt:=\lim_{\eta \uparrow b}\smallint\limits_0^\eta   G_\t (x,t,\l)f(t)\, dt,\;\; f=f(\cd)\in\gH .
\end{equation}
\end{theorem}
\begin{proof} \emph{Step1.} Let $\gH_b$ be the set of all functions $f\in\gH$ such that $f(t)=0$ on some interval $(\eta, b)\subset\D$ (depending on $f$). First we show that
\begin{equation}\label{2.26}
((A_0-\l)^{-1}f)(x)=\smallint\limits_0^b  G_0 (x,t,\l)f(t)\, dt,\;\;\;
\; f=f(\cd)\in\gH_b, \quad\; \l\in\bC_+ ,
\end{equation}
where $A_0$ is the symmetric extension \eqref{1.33} and the operator kernel $G_0 (x,t,\l)$ is
\begin{equation}\label{2.27}
G_0 (x,t,\l)=\begin{cases} -v_0(x,\l)\, c^*(t,\ov\l), \;\; x>t \cr -c(x,\l)\, v_0^* (t,\ov\l), \;\; x<t \end{cases}, \quad \l\in\bC_+\cup\bC_-.
\end{equation}
To prove  \eqref{2.26} it is sufficient to show that for every $f=f(t)\in\gH_b$ the function
\begin{equation}\label{2.28}
y=y(x,\l):=\smallint\limits_0^b  G_0 (x,t,\l)f(t)\, dt, \quad \l\in\bC_+
\end{equation}
belongs to $\cD (A_0)$ and obeys the equality $l[y]-\l y=f$.

It follows from \eqref{2.28} and \eqref{2.27} that
\begin{equation}\label{2.29}
y=y(x,\l)=c(x,\l)C_1(x)+v_0(x,\l)C_2(x)=Y_v(x,\l)C(x), \quad \l\in\bC_+
\end{equation}
where
\begin{align}
C_1(x)=-\smallint\limits_x^b v_0^*(t,\ov \l)f(t)\, dt, \qquad  C_2(x)=-\smallint\limits_0^x c^*(t,\ov \l)f(t)\, dt, \qquad\label{2.30}\\
Y_v(x,\l)=(c(x,\l)\;\;v_0(x,\l)):H^n\oplus H^n\to H, \quad \l\in\bC_+\cup \bC_-,\label{2.31}\\
C(x)=\{C_1(x),C_2(x)\}\;(\in H^n\oplus H^n).\qquad\qquad\label{2.32}
\end{align}
Next we show that the functions \eqref{2.31}, \eqref{2.32} satisfy hypothesis of Lemma \ref{lem1.6}.

In view of \eqref{1.37a} and \eqref{1.37b} $Y_v(\cd,\l)$ is the solution of \eqref{1.23} with the initial data
\begin{equation}\label{2.32a}
\wt Y_v(0,\l)=\begin{pmatrix} c^{(1)}(0,\l) & v_0^{(1)}(0,\l) \cr   c^{(2)}(0,\l) & v_0^{(2)}(0,\l)\end{pmatrix}= \begin{pmatrix} I_{H^n} & -m(\l)\cr 0 & I_{H^n} \end{pmatrix}, \quad \l\in\bC_+ \cup \bC_-.
\end{equation}
Hence $0\in\rho (\wt Y_v (0,\l))$ and, consequently,  $0\in\rho (\wt Y_v (x,\l))$ for all $x\in\D$. Next the direct calculation with taking the relation $m^*(\ov\l)=m(\l)$ into account   gives $\wt Y_v^*(0,\ov\l)J_{H^n} \wt Y_v(0,\l)=J_{H^n}$, where $J_{H^n}$ is the operator \eqref{1.0.0}.
Moreover the Lagrange's identity \eqref{1.28} implies
\begin{equation*}
\wt Y_v^*(x,\ov\l)J_{H^n} \wt Y_v(x,\l)=\wt Y_v^*(0,\ov\l)J_{H^n} \wt Y_v(0,\l)\,(=J_{H^n}), \qquad x\in\D.
\end{equation*}
Therefore in view of the invertibility of $\wt Y_v (x,\l)$  one has
$\wt Y_v(x,\l)J_{H^n} \wt Y_v^*(x,\ov\l)=J_{H^n}$,
which is equivalent to the relations
\begin{align}
c^{(1)}(x,\l)v_0^{(1)*}(x,\ov\l)- v_0^{(1)}(x,\l)c^{(1)*}(x,\ov\l)=0,\qquad\qquad \qquad\qquad \label{2.36}\\
c^{(2)}(x,\l)v_0^{(2)*}(x,\ov\l)- v_0^{(2)}(x,\l)c^{(2)*}(x,\ov\l)=0, \qquad\qquad\qquad\qquad  \label{2.37}\\
c^{(2)}(x,\l)v_0^{(1)*}(x,\ov\l)- v_0^{(2)}(x,\l)c^{(1)*}(x,\ov\l)=-I_{H^n},\quad x\in\D, \;\;\l\in \bC_+\cup \bC_- \label{2.38}
\end{align}
It follows from \eqref{2.36}, \eqref{2.38} that
\begin{align*}
c^{(1)}(x,\l)v_0^{*}(x,\ov\l)- v_0^{(1)}(x,\l)c^*(x,\ov\l)=0,\qquad\qquad \qquad\qquad\qquad\qquad\\
c^{(2)}(x,\l)v_0^{*}(x,\ov\l)- v_0^{(2)}(x,\l)c^*(x,\ov\l)=(-I_H \;\; 0\;\;\dots\;\;0 )^\top, \quad x\in\D, \;\;\l\in \bC_+\cup \bC_-.
\end{align*}
Moreover the equalities \eqref{2.30}, \eqref{2.32} give
\begin{equation*}
C'(x)= (v_0^{*}(x,\ov\l)\;\; -c^*(x,\ov\l))^\top f(x)\;\;(\text{mod}\;\;\mu).
\end{equation*}
Hence nearly everywhere on $\D$ one has
\begin{equation*}
\wt Y_v(x,\l)C'(x)=\begin{pmatrix} c^{(1)}(x,\l) & v_0^{(1)}(x,\l) \cr   c^{(2)}(x,\l) & v_0^{(2)}(x,\l)\end{pmatrix}\begin{pmatrix}v_0^{*}(x,\ov\l)\cr -c^*(x,\ov\l)  \end{pmatrix} f(x)=\{\underbrace{0,\dots, 0,}_n  -f(x),\underbrace {0,\dots, 0}_{n-1}\},
\end{equation*}
which coincides with \eqref{1.26}. Therefore according to Lemma \ref{lem1.6} the function \eqref{2.29} belongs to $\cD (l)$ and obeys the relations
\begin{equation}\label{2.41}
l[y]-\l y=f, \qquad \wt y(x,\l)=\wt Y_v(x,\l)C(x).
\end{equation}
Since $f\in\gH_b$, it follows from \eqref{2.30} that $C_1(x)\equiv 0$ and $C_2(x)\equiv C_2\,(\in H^n)$ on some interval $(\eta,b)\subset\D$. Hence by \eqref{2.29} $y=v_0(x,\l)C_2,\; x\in (\eta,b)$, which yields the inclusion $y\in\cD$. Moreover  according to \cite{Mog} (see proof of Lemma 3.4) $\G'_0 y=\G'_0 (v_0(x,\l)C_2)$ and  by \eqref{1.35}  $\G'_0 y=0$. Finally combining the second equality in \eqref{2.41} with \eqref{2.32a} and \eqref{2.32}, we obtain $y^{(2)}(0,\l)=C_2(0)=0$.

Thus the function \eqref{2.28} satisfies the boundary conditions in the right hand part of \eqref{1.33} and, consequently, it belongs to $\cD (A_0)$.

\emph{Step2.} Next by using formula for resolvents \eqref{1.17} we prove \eqref{2.24} for all $f=f(t)\in\gH_b$ and $\l\in\rho (\wt A)\cap \bC_+$. Applying Lemma \ref{lem1.8}, 3) to the  solution $Z_-(\cd,\ov\l)$  and taking \eqref{1.41} into account one obtains
\begin{equation*}
\g_-^*(\ov\l)f=\smallint\limits_0^b Z_-^*(t,\ov\l)f(t)\, dt, \qquad \l\in\bC_+, \;\;\; f\in\gH_b.
\end{equation*}
This and \eqref{1.40} yield
\begin{align*}
\bigl ( \g_+(\l)(C_0+C_1M_+(\l))^{-1}C_1 \g_-^*(\ov\l)f \bigr )(x)=\qquad\qquad\qquad\qquad\\
\smallint\limits_0^b Z_+(x,\l)(C_0+C_1M_+(\l))^{-1}C_1Z_-^*(t,\ov\l)f(t)\, dt=
\smallint\limits_0^b G_1(x,t,\l)f(t)\, dt,
\end{align*}
where
\begin{equation}\label{2.42}
G_1(x,t,\l)=Z_+(x,\l)(C_0+C_1M_+(\l))^{-1}C_1Z_-^*(t,\ov\l),\quad \l\in\rho (\wt A)\cap \bC_+.
\end{equation}
Moreover the resolvent $(A_0-\l)^{-1}$ is defined by \eqref{2.26}. Hence formula \eqref{1.17} for the decomposing $D$-triplet $\Pi$ takes the form
\begin{equation}\label{2.43}
((\wt A-\l)^{-1}f)(x)=\smallint\limits_0^b  G(x,t,\l)f(t)\, dt,\;\;\; f=f(\cd)\in\gH_b, \;\;\; \l\in\rho (\wt A)\cap \bC_+,
\end{equation}
where $ G(x,t,\l)=G_0(x,t,\l)-G_1(x,t,\l), \; \l\in\rho (\wt A)\cap\bC_+$. Now it remains to show that $G(x,t,\l)$ coincides with the Green function $G_\t (x,t,\l)$ (see \eqref{2.22}).

Denote by $G_+(x,t,\l)$ and $G_-(x,t,\l)$ restrictions of $G(x,t,\l)$ onto the sets $\{(x,t): x>t\}$ and $\{(x,t): x<t\}$ respectively. It follows from \eqref{2.27} and \eqref{2.42} that
\begin{align}
G_+(x,t,\l)=-v_0(x,\l)c^*(t,\ov\l)-Z_+(x,\l)(C_0+C_1M_+(\l))^{-1}
C_1Z_-^*(t,\ov\l),
\label{2.44}\\
G_-^*(x,t,\l)=-v_0(t,\ov\l)c^*(x,\l)-Z_-(t,\ov\l)C_1^*(C_0^*
+M_-(\ov\l)C_1^*)^{-1}
Z_+^*(x,\l).\label{2.45}
\end{align}
Next assume that $\t^\times=\{(C_{0\times},C_{1\times});\cK_\times\} \in\CA$ is the $\times$-adjoint operator pair  \eqref{2.4}  and let $\t^*=\{(C_{1*},C_{0*}); \cK_\tm\}\in \CB$ be the adjoint pair with $C_{1*}$ and $C_{0*}$ given by  \eqref{1.5}. It follows from \eqref{2.4} that the block-matrix representations
\begin{equation}\label{2.46}
C_{1*}=(\hat C_{2\times}\;\; C_{1*}'):H^n\oplus\cH_1'\to\cK_\times, \qquad  C_{0*}=(-\hat C_{1\times}\;\; C_{0*}'):H^n\oplus\cH_0'\to\cK_\times.
\end{equation}
are valid with $C_{1*}'=C_{0\tm}'\up \cH_1'$ and $C_{0*}'=C_{1\tm}'P_1'-iC_{0\tm}'P_2'$. Now by using Lemma \ref{lem0.1} we can rewrite \eqref{2.45} as
\begin{equation}\label{2.47}
G_-^*(x,t,\l)=-v_0(t,\ov\l)c^*(x,\l)-Z_-(t,\ov\l)(C_{1*}+ C_{0*} M_-(\ov\l))^{-1}C_{0*} Z_+^*(x,\l).
\end{equation}
Let
\begin{equation}\label{2.48}
Y_1(t,\ov\l):=(-c(t,\ov\l)\;\;0):H^n\oplus\cH_0'\to H, \;\; Y_2(t,\l):=(-c(t,\l)\;\;0):H^n\oplus\cH_1'\to H
\end{equation}
and let $Y_-(\cd, \ov\l):\D\to [\cH_0,H], \; Y_+(\cd, \l):\D\to [\cH_1,H]$ be operator solutions given by
\begin{align}
Y_-(t,\ov\l):=Y_1(t,\ov\l) - Z_-(t,\ov\l)  C_1^*(C_0^*+M_-(\ov\l)C_1^*) ^{-1}\quad\label{2.49}\\
Y_+(t,\l):=Y_2(t,\l) - Z_+(t,\l)  C_{0*}^*(C_{1*}^*+M_+(\l) C_{0*}^*)^{-1} \label{2.50}
\end{align}
Combining \eqref{2.48} with \eqref{1.38} and \eqref{1.39} one obtains
\begin{equation*}
-v_0(x,\l)c^*(t,\ov\l)=Z_+(x,\l)Y_1^*(t,\ov\l), \qquad -v_0(t,\ov\l)c^*(x,\l)=Z_-(t,\ov\l)Y_2^*(x,\l).
\end{equation*}
Hence the equalities \eqref{2.44} and  \eqref{2.47} can be represented as \begin{equation}\label{2.51}
G_+(x,t,\l)=Z_+(x,\l)Y_-^*(t,\ov\l), \qquad G_-^*(x,t,\l)=Z_-(t,\ov\l) Y_+^*(x,\l)
\end{equation}
It follows from  \eqref{1.40} and  \eqref{1.41} that the corresponding operators  \eqref{2.11} and  \eqref{2.19} for  $Z_+(\cd,\l)$ and $Z_-(\cd,\ov\l)$ are
\begin{equation*}
T=(C_0\G_0+C_1\G_1)\g_+(\l), \qquad T_\tm=(C_{0\tm}\G_0+ C_{1\tm}\G_1)\g_-(\ov\l).
\end{equation*}
Therefore by Lemma \ref{lem1.2} one has
\begin{equation}\label{2.52}
T=C_0+C_1M_+(\l), \qquad T_\tm= C_{1*}+C_{0*}M_-(\ov\l).
\end{equation}
Next by  \eqref{2.49} $Y_-(\cd, \ov\l)$ is the operator solution if the equation $l[y]-\ov\l y=0$ and
\begin{equation*}
\wt Y_-(0,\ov\l)=\wt Y_1(0,\ov\l)-\wt Z_-(0,\ov\l) C_1^* (C_0^*+M_-(\ov\l)C_1^*)^{-1}=X(C_0^*+M_-(\ov\l)C_1^*)^{-1},
\end{equation*}
where
\begin{align*}
X=\wt Y_1(0,\ov\l)(C_0^*+M_-(\ov\l)C_1^*)- \wt Z_-(0,\ov\l)C_1^*=\qquad\qquad\qquad\qquad\qquad\qquad\\
\begin{pmatrix}-I&0\cr0&0 \end{pmatrix}
\left[ \begin{pmatrix} \hat C_2^* \cr {C'_0}^* \end{pmatrix}
+\begin{pmatrix} m(\ov\l) & M_{2-}(\ov\l) \cr  M_{3-}(\ov\l) & M_{4-}(\ov\l) \end{pmatrix} \begin{pmatrix} -\hat C_1^* \cr {C'_1}^* \end{pmatrix}\right] - \begin{pmatrix} -m(\ov \l) & -M_{2-}(\ov\l) \cr  I_{H^n} & 0 \end{pmatrix} \begin{pmatrix} -\hat C_1^* \cr {C'_1}^* \end{pmatrix}=\\
\begin{pmatrix} -\hat C_2^* +m(\ov\l)\hat C_1^* -M_{2-}(\ov\l){C'_1}^*  \cr 0 \end{pmatrix}-\begin{pmatrix} m(\ov\l)\hat C_1^* -M_{2-}(\ov\l){C'_1}^*  \cr -\hat C_1^* \end{pmatrix}=\begin{pmatrix} -\hat C_2^* \cr  \hat C_1^*\end{pmatrix}\qquad\qquad
\end{align*}
(here we made use of the block-matrix representations \eqref{2.3}, \eqref{1.32b} and \eqref{1.37b}). This implies that the solution $Y_-(\cd, \ov \l)$ satisfies the initial data
\begin{equation}\label{2.53}
\wt Y_-(0,\ov\l)=(-\hat C_2^*\;\;\;\hat C_1^*)^\top\,  (C_0^*+M_-(\ov\l)C_1^*)^{-1}=(-\hat C_2^*\;\;\;\hat C_1^*)^\top\, T^{-1*}
\end{equation}
Similar calculations for $Y_+(t,\l)$ (with taking \eqref{2.46} into account) gives
\begin{equation}\label{2.55}
\wt Y_+(0,\l)=(-\hat C_{2\tm}^*\;\;\hat C_{1\tm}^*)^\top \, (C_{1*}^*+M_+(\l)C_{0*}^*)^{-1}= (-\hat C_{2\tm}^*\;\;\hat C_{1\tm}^*)^\top \, T_\tm^{-1*}.
\end{equation}
Comparing \eqref{2.53} and \eqref{2.55} with \eqref{2.18} and \eqref{2.20}, one obtains $Y_-(t,\ov\l)=Y_\t(t,\ov\l)$ and $Y_+(t,\l)=Y_{\t^\tm}(t,\l)$. Now the  equality $G(x,t,\l)=G_\t (x,t,\l)$ is implied by \eqref{2.51}.

\emph{Step 3.} To complete the proof it is necessary to extend the above result to all $f\in\gH$ and $\l\in\rho (\wt A)$.

If $\l\in \rho (\wt A)\cap\bC_-$, then $\ov\l\in \rho (\wt A^*)\cap\bC_+$ and, consequently,
\begin{equation}\label{2.57}
((\wt A^*-\ov\l)^{-1}f)(x)=\smallint\limits_0^b G_{\t^\times}(x,t,\ov\l)f(t)\, dt, \qquad f=f(\cd)\in\gH_b.
\end{equation}
Since $(\wt A-\l)^{-1}=((\wt A^*-\ov\l)^{-1})^*$, it follows from \eqref{2.57} that  $(\wt A-\l)^{-1}|\gH_b$ is the integral operator with the kernel $G'(x,t,\l)=(G_{\t^\times}(t,x,\ov\l))^*=G_\t(x,t,\l)$
(here the second equality is immediate from \eqref{2.22}). Therefore \eqref{2.24} holds for all $f\in\gH_b$ and $\l\in\rho (\wt A)\cap\bC_-$.

Next assume that  $\l_0\in\rho (\wt A)\cap\bR$. Then there exists a disk $U(\l_0)=\{\l\in\bC : |\l\-\l_0|<\varepsilon \} \; (\varepsilon > 0 )$ such that $U(\l_0)\subset \rho (\wt A)$. Let $\cZ(\cd)$ be an operator function \eqref{29.5} and let $\cZ(\cd, \l)\; (\l\in\rho (\wt A))$ be a family of fundamental solutions \eqref{29.6} (see the proof of Theorem \ref{th1.9a}). Since by Proposition 3.9 in \cite{Mog06.2} $\G_j\in [\cD_+, \cH_j], \; j\in\{0,1\}$, the operator functions
\begin{equation*}
T(\l)=(C_0\G_0+C_1\G_1) \cZ(\l), \quad  T_\tm(\l)=(C_{0\tm}\G_0+C_{1\tm} \G_1) \cZ(\l), \quad \l\in U(\l_0)
\end{equation*}
are holomorphic on $U(\l_0)$. Let now $\cY (\cd,\l):\D\to [\cK',H]$ and $\cY_\tm(\cd,\l):\D\to [\cK',H]$ be operator solutions of \eqref{1.23} with
\begin{equation*}
\wt\cY (0,\l)=(-\hat C_2^*\;\;\;\hat C_1^*)^\top\, (T(\ov\l))^{-1*}, \quad \wt\cY_\tm (0,\l)=(-\hat C_{2\tm}^*\;\;\;\hat C_{1\tm}^*)^\top\, (T_\tm(\ov\l))^{-1*}, \quad \l\in U(\l_0)
\end{equation*}
Then for every  $\l\in U(\l_0)$ the corresponding Green function  \eqref{2.22} can be written as
\begin{equation}\label{2.58}
G_\t (x,t,\l)=\begin{cases} \cZ(x,\l)\, \cY^*(t,\ov\l), \;\; x>t \cr \cY_\tm(x,\l)\, \cZ^* (t,\ov\l), \;\; x<t \end{cases}.
\end{equation}
Since all initial data $\wt\cZ (0,\l),\wt\cY (0,\l)$ and $\wt\cY_\tm (0,\l)$ are continuous (and even holomorphic) at the point $\l_0$, it follows from \eqref{2.58} that $\lim\limits_{\l\to\l_0}\sup\limits_ {(x,t)\in R}||G_\t(x,t,\l)- G_\t(x,t,\l_0)||=0$ for each closed rectangle $R\subset\D\times\D$. Therefore for every $f=f(\cd)\in\gH_b$ we can   pass to the limit in the equality \eqref{2.24} as $\bC_+\ni\l\to\l_0$, which gives the same equality for  $\l=\l_0$. Hence \eqref{2.24} holds for all $\l\in\rho (\wt A)$ and $f\in\gH_b$.

Finally \eqref{29.2} implies existence of the limit $\lim_{\eta \uparrow b}\smallint\limits_0^\eta   G_\t (x,t,\l)f(t)\, dt$ for all $f(\cd)\in\gH$ , which completely  proves \eqref{2.24}.
\end{proof}
\section{Generalized resolvents and characteristic matrices of differential operators}
\subsection{Generalized resolvents and characteristic matrices}
In this section by using the concept of a decomposing $D$-triplet we complement the known results on generalized resolvents and characteristic matrices of differential operators \cite{Sht57,Bru74}.

Assume that $\Pi=\bta$ is a decomposing $D$-triplet \eqref{1.30} for $L$ and $\g_\pm(\cd),\; M_\pm(\cd)$ are the corresponding $\g$-fields
and Weyl functions (see \eqref{1.13}--\eqref{1.15}). Moreover let $\pair\in\RH$ be a collection of two holomorphic operator pairs
\begin{equation}\label{3.0}
\tau_+(\l)=\{(C_0(\l),C_1(\l));\cK_+\}, \;\l\in\bC_+; \;\; \tau_-(\l)=\{(D_0(\l),D_1(\l));\cK_-\}, \;\l\in\bC_-
\end{equation}
with the block-matrix representations
\begin{align}
C_0(\l)=(\hat C_{2}(\l)\;\;C'_{0}(\l)): H^n\oplus\cH_0'\to\cK_+, \qquad \qquad\qquad \qquad\qquad \qquad\qquad  \label{3.1}\\
\qquad \qquad\qquad \qquad \qquad \qquad\qquad \quad C_1(\l)=(\hat C_{1}(\l)\;\;C'_1(\l)):H^n\oplus \cH_1'\to\cK_+ , \quad \l\in\bC_+\nonumber\\
\qquad D_0(\l)=(\hat D_{2}(\l)\;\;D'_{0}(\l)):H^n\oplus\cH_0'\to\cK_-,\qquad \qquad\qquad \qquad\qquad \qquad\qquad\label{3.2}\\
\qquad \qquad\qquad \qquad \qquad \qquad\qquad \qquad D_1(\l)=(\hat D_{1}(\l)\;\;D'_1(\l)):H^n\oplus \cH_1'\to\cK_-, \quad \l\in\bC_-\nonumber
\end{align}
Consider the operator functions $\wt T_+(\l), \; \l\in\bC_+ $ and   $\wt T_-(\l),\; \l\in\bC_-$ given by
\begin{equation}\label{3.3}
\wt T_+(\l)=(C_0(\l)\G_0 - C_1(\l)\G_1)\g_+(\l), \quad \wt T_-(\l)=(D_0(\l)\G_0 - D_1(\l)\G_1)\g_-(\l).
\end{equation}
Moreover introduce the  operator pair $\wt \tau_-(\l)=\{(\wt D_1(\l), \wt D_0(\l));\cK_-\}, \;\l\in\bC_-$ via
\begin{align}
\wt D_1(\l)=D_0(\l)\up \cH_1=(\hat D_2(\l)\;\; D_0'(\l)\up\cH_1'):H^n\oplus \cH_1'\to \cK_-,\qquad\qquad \qquad\label{3.4}\\
\wt D_0(\l)=D_1(\l)P_1+iD_0(\l)P_2=(\hat D_1(\l)\;\; D_1'(\l)P_1'+iD_0'(\l)P_2'):H^n\oplus \cH_0'\to \cK_-.\label{3.5}
\end{align}
Since $(-\tau_+(\ov\l))^\tm=-\tau_-(\l)=\{(D_0(\l),-D_1(\l));\cK_-\}$, it follows from \eqref{1.5} that
\begin{equation}\label{3.6}
(-\tau_+(\ov\l))^*=\{(\wt D_1(\l),-\wt D_0(\l));\cK_-\}=-\wt \tau_-(\l), \quad \l\in\bC_-
\end{equation}
This and Lemma \ref{lem1.2} yield
\begin{equation}\label{3.7}
\wt T_+(\l)=C_0(\l)-C_1(\l)M_+(\l), \;\;\l\in\bC_+; \;\;\; \wt T_-(\l)=\wt D_1(\l)-\wt D_0(\l)M_-(\l), \;\;\l\in\bC_-
\end{equation}
The following theorem, which is immediate from Theorem \ref{th1.4}, contains the description of all generalized resolvents of the minimal operator $L_0$ in terms of boundary conditions.
\begin{theorem}\label{th3.1}
Let  $\Pi=\bta$ be a decomposing $D$-triplet \eqref{1.30} for $L$. Then for every collection $\pair\in\RH$, given by \eqref{3.0} -- \eqref{3.2}, the equality
\begin{multline}\label{3.8}
\cD (\wt A(\l))=\\
\qquad\;\; =\begin{cases}\{y\in\cD: \,\hat C_1(\l) y^{(1)}(0)+ \hat C_2(\l) y^{(2)}(0)+C'_0(\l)\G'_0 y-C'_1(\l)\G'_1 y=0\},\;\; \;\l\in\bC_+ \cr \{y\in\cD: \,\hat D_1(\l) y^{(1)}(0)+ \hat D_2(\l) y^{(2)}(0)+D'_0(\l)\G'_0 y-D'_1(\l)\G'_1 y=0\},\;\; \;\l\in\bC_- \end{cases}
\end{multline}
defines a family of proper extensions $\wt A(\l)=L\up \cD (\wt A(\l))$ such that the operator function
\begin{equation}\label{3.9}
\R=\bR_\tau (\l):= (\wt A(\l)-\l)^{-1},\quad \l\in\bC_+\cup \bC_-
\end{equation}
is a generalized resolvent of $L_0$. Conversely for every generalized resolvent $\R$ of the operator $L_0$ there exists the unique collection $\pair\in\RH$ such that \eqref{3.8} and \eqref{3.9} hold. Moreover $\R$ is a canonical resolvent if and only if $\tau\in\RZ$.
\end{theorem}
By using Theorem \ref{th2.8} one can represent the above result in a rather different form. Namely, let $Z_\pm (\cd,\l)$ be fundamental solutions \eqref{1.38}, \eqref{1.39} and let $T_+=T_+(\l)\;\; (T_-=T_-(\l))$ be the operator  \eqref{2.11}, corresponding to the operator pair $-\tau_+(\l)=$ $\{(C_0(\l),$ $-C_1(\l));\cK_+\}$ (respectively $-\tau_-(\l)= \{(D_0(\l),-D_1(\l)) ; \cK_-\}$ for a fixed $\l\in\bC_+\; (\l\in\bC_-)$. Then in view of \eqref{1.40} and \eqref{1.41} $T_\pm(\l)=\wt T_\pm(\l)$, where $\wt T_\pm(\l)$ are defined by \eqref{3.3}.

Next assume that $\wt A(\l)$ is a family of extensions \eqref{3.8}. Then $\l\in\rho (\wt A(\l))$ and according to Theorem \ref{th2.3} $0\in\rho (\wt T_\pm(\l)),\;\l\in\bC_\pm$. This and \eqref{3.7} make it possible to associate with every collection $\pair\in\RH$, given by \eqref{3.0} -- \eqref{3.2}, the families of operator solutions $Y_{\tau +}(\cd,\l):\D \to [\cH_1,H], \;\l\in\bC_+$ and $Y_{\tau -}(\cd,\l):\D \to [\cH_0,H], \;\l\in\bC_-$ of the equation \eqref{1.23} defined by the initial data
\begin{align}
\wt Y_{\tau +}(0,\l)=(-\hat D_2^*(\ov\l)\;\;\; \hat D_1^*(\ov\l))^\top \, (\wt D_1^*(\ov\l)-M_+(\l)\wt D_0^*(\ov\l))^{-1}, \quad \l\in\bC_+ \label{3.10}\\
\wt Y_{\tau -}(0,\l)=(-\hat C_2^*(\ov\l)\;\;\; \hat C_1^*(\ov\l))^\top \, (C_0^*(\ov\l)-M_-(\l) C_1^*(\ov\l))^{-1}, \quad \l\in\bC_-.\label{3.11}
\end{align}
Moreover introduce the operator functions
\begin{equation}\label{3.12}
Y_\tau (t,\l)=\begin{cases} Y_{\tau +}(t,\l), \;\l\in\bC_+ \cr  Y_{\tau -}(t,\l), \;\l\in\bC_- \end{cases}; \quad  Z_0(t,\l)=\begin{cases} Z_+(t,\l), \;\l\in\bC_+ \cr  Z_ -(t,\l), \;\l\in\bC_-. \end{cases}
\end{equation}
One can easily verify that $Y_\tau(\cd,\l)$ is uniquely defined by $\tau$ and does not depend on (equivalent) representations \eqref{3.0}.
\begin{definition}\label{def3.2}
The operator function $G_\tau (\cd,\cd,\l):\D\tm\D\to [H]$, given by
\begin{equation}\label{3.13}
G_\tau (x,t,\l)=\begin{cases} Z_0(x,\l)\, Y_\tau^*(t,\ov\l), \;\; x>t \cr Y_{\tau}(x,\l)\, Z_0^* (t,\ov\l), \;\; x<t \end{cases}, \quad \l\in\bC_+\cup\bC_-
\end{equation}
will be called the Green function, corresponding to a collection $\pair\in\RH$.
\end{definition}
Now combining statements of Theorems \ref{th2.8} and \ref{th3.1} we arrive at the following theorem.
\begin{theorem}\label{th3.3}
Let under conditions of Theorem \ref{th3.1} $\tau\in\RH$ and let $\R=\bR_\tau(\l)$ be the corresponding generalized resolvent given by \eqref{3.8} and \eqref{3.9}. Then
\begin{equation}\label{3.14}
(\R f)(x)=\smallint\limits_0^b  G_\tau (x,t,\l)f(t)\, dt:=\lim_{\eta \uparrow b}\smallint\limits_0^\eta   G_\tau (x,t,\l)f(t)\, dt,\;\; f=f(\cd)\in\gH,
\end{equation}
which implies that the equalities \eqref{3.13} and \eqref{3.14}  establish a bijective correspondence between all generalized  (canonical) resolvents $\R$ of the minimal operator $L_0$ and all collections $\tau\in\RH$ (respectively, $\tau\in\RZ$).
\end{theorem}
Assume that $\Pi=\bta$ is a decomposing $D$-triplet \eqref{1.30} for $L$, $M_\pm(\cd)$ are the corresponding Weyl functions and $\pair\in\RH$ is a collection of holomorphic operator pairs given by \eqref{3.0} or, equivalently, by \eqref{1.6}. Then according to  \cite{Mog06.2} $0\in\rho (\tau_+(\l)+M_+(\l))$, which in view of \cite{MalMog02}, Lemma 2.1 is equivalent to each of the inclusions $0\in \rho (C_0(\l)-C_1(\l)M_+(\l))$ and $0\in\rho (K_1(\l)+M_+(\l)K_0(\l)), \;\l\in\bC_+$. Moreover the same lemma in \cite{MalMog02} yields
\begin{align}
-(\tau_+(\l)+M_+(\l))^{-1}=(C_0(\l)-C_1(\l)M_+(\l))^{-1}C_1(\l)= \label{3.17}\\
=-K_0(\l) (K_1(\l)+M_+(\l)K_0(\l))^{-1},\quad \l\in\bC_+\qquad\nonumber
\end{align}

With every collection $\pair\in\RH$ we associate two holomorphic operator functions $\wt\Om_{\tau +}(\cd):\bC_+\to [\cH_0\oplus\cH_1, \cH_1\oplus\cH_0] $ and $\wt\Om_{\tau -}(\cd):\bC_-\to [\cH_1\oplus\cH_0, \cH_0\oplus\cH_1] $. The first of them is given by the block-matrix representation
\begin{equation}\label{3.18}
\wt\Om_{\tau +}(\l)=\begin{pmatrix}\wt\om_{1+}(\l) & \wt\om_{2+}(\l) \cr  \wt\om_{3+}(\l) & \wt\om_{4+}(\l) \end{pmatrix}:\cH_0\oplus\cH_1\to \cH_1\oplus\cH_0, \quad \l\in\bC_+,
\end{equation}
with entries
\begin{align}
\wt\om_{1+}(\l)=M_+(\l)-M_+(\l)(\tau_+(\l) +M_+(\l))^{-1}M_+(\l)\qquad\qquad\qquad \label{3.18a}\\
\wt\om_{2+}(\l)= -\tfrac 1 2 I_{\cH_1} +M_+(\l) (\tau_+(\l) +M_+(\l))^{-1}\qquad\qquad\qquad\qquad\label{3.18b}\\
\wt\om_{3+}(\l)=-\tfrac 1 2 I_{\cH_0} +(\tau_+(\l) +M_+(\l))^{-1}M_+(\l), \;\;\;\wt\om_{4+}(\l)=-(\tau_+(\l) +M_+(\l))^{-1},\label{3.18c}
\end{align}
while the second one is defined by the relation
\begin{equation}\label{3.19}
\wt\Om_{\tau_-}(\l)=(\wt\Om_{\tau +}(\ov\l))^*,\quad \l\in\bC_-.
\end{equation}
It follows from \eqref{3.17} that
\begin{align}
(\tau_+(\l)+M_+(\l))^{-1}M_+(\l)=-(C_0(\l)-C_1(\l)M_+(\l))^{-1}
C_1(\l)M_+(\l) =\label{3.19a}\\
=I_{\cH_0}-(C_0(\l)-C_1(\l)M_+(\l))^{-1}C_0(\l), \quad \l\in\bC_+.\qquad\quad\nonumber
\end{align}
This and \eqref{3.17} imply that formulas \eqref{3.18a}-- \eqref{3.18c} can be represented as
\begin{align}
\wt\om_{1+}(\l)=M_+(\l)(C_0(\l)-C_1(\l)M_+(\l))^{-1} C_0(\l) \qquad \label{3.20a}\\
\wt\om_{2+}(\l)=-\tfrac 1 2 I_{\cH_1} -M_+(\l) (C_0(\l)-C_1(\l)M_+(\l))^{-1} C_1(\l)\label{3.20b}\\
\wt\om_{3+}(\l)=\tfrac 1 2 I_{\cH_0}-(C_0(\l)-C_1(\l)M_+(\l))^{-1} C_0(\l),\quad\label{3.20c}\\
\wt\om_{4+}(\l)=(C_0(\l)-C_1(\l)M_+(\l))^{-1} C_1(\l).\qquad\quad\label{3.20d}
\end{align}
Similar arguments for the operator function $\wt\Om_{\tau_-}(\cd)$ with taking \eqref{3.6} into account gives
\begin{equation}\label{3.21}
\wt\Om_{\tau -}(\l)=\begin{pmatrix}\wt\om_{1-}(\l) & \wt\om_{2-}(\l) \cr  \wt\om_{3-}(\l) & \wt\om_{4-}(\l) \end{pmatrix}:\cH_1\oplus\cH_0\to \cH_0\oplus\cH_1, \quad \l\in\bC_-,
\end{equation}
where
\begin{align}
\wt\om_{1-}(\l)=M_-(\l)(\wt D_1(\l)-\wt D_0(\l)M_-(\l))^{-1} \wt D_1(\l) \qquad \label{3.21a}\\
\wt\om_{2-}(\l)=-\tfrac 1 2 I_{\cH_0} -M_-(\l) (\wt D_1(\l)-\wt D_0(\l)M_-(\l))^{-1} \wt D_0(\l)\label{3.21b}\\
\wt\om_{3-}(\l)=\tfrac 1 2 I_{\cH_1}-(\wt D_1(\l)-\wt D_0(\l)M_-(\l))^{-1} \wt D_1(\l),\quad\label{3.21c}\\
\wt\om_{4-}(\l)=(\wt D_1(\l)-\wt D_0(\l)M_-(\l))^{-1} \wt D_0(\l).\qquad\quad\label{3.21d}
\end{align}
and the operator functions $\wt D_1(\cd)$ and $\wt D_0(\cd)$ are defined by \eqref{3.4} and \eqref{3.5} respectively.

Since $H^n\in\cH_0$ and $H^n\in\cH_1$, it follows that $H^n\oplus H^n\subset \cH_0\oplus\cH_1$ and $H^n\oplus H^n\subset \cH_1\oplus\cH_0$. Therefore with every collection $\pair\in\RH$ we can associate the operator function $\Om_\tau (\cd): \bC_+\cup\bC_-\to [H^n\oplus H^n]$ given by
\begin{equation}\label{3.22}
\Om_\tau (\l)=\begin{cases} P_{H^n\oplus H^n}\,\wt \Om_{\tau +}(\l)\up H^n\oplus H^n, \;\; \l\in\bC_+ \\
P_{H^n\oplus H^n}\,\wt \Om_{\tau_-}(\l)\up H^n\oplus H^n, \;\; \l\in\bC_- \end{cases}
\end{equation}
It follows from \eqref{3.19} and \eqref{3.22} that $\Om_\tau(\cd)$ is a holomorphic operator function obeying the relation $\Om_\tau^*(\ov\l)=\Om_\tau(\l),\;\l\in\bC_+\cup\bC_-$.
\begin{theorem}\label{th3.4}
Let $\Pi=\bta$ be a decomposing $D$-triplet \eqref{1.30} for $L$. Then for every collection $\pair\in\RH$ the corresponding Green function \eqref{3.13} admits the representation
\begin{equation}\label{3.23}
G_\tau(x,t,\l)=Y_0(x,\l)(\Om_\tau(\l)+\tfrac 1 2 \, \emph{sgn}(t-x)J_{H^n})Y_0^*(t,\ov\l), \quad \l\in\bC_+\cup\bC_-,
\end{equation}
where $Y_0(\cd,\l)$ is the "canonical" solution \eqref{1.25} and $J_{H^n}$ is the operator \eqref{1.0.0}.
\end{theorem}
\begin{proof}
Let $M_\pm(\cd)$ be Weyl functions \eqref{1.32a}, \eqref{1.32b} and let
\begin{align}
\wt S_+(\l):=(-M_+(\l)\;\;  I_{\cH_0})^\top :\cH_0
\to \cH_1 \oplus \cH_0,
 \;\;\;\l\in\bC_+\label{3.23.1}\\
\wt S_-(z):=(-M_-(z)\;\;  I_{\cH_1})^\top :\cH_1\to \cH_0\oplus\cH_1, \;\;\; z\in\bC_-.\label{3.23.2}
\end{align}
Introduce the  operator functions
$S_+(\l)(\in [\cH_0,H^n\oplus H^n]),\; \l\in\bC_+$ and $S_-(z)(\in
 [\cH_1,H^n\oplus H^n]),\; z\in\bC_-$ by
\begin{align}
S_+(\l)=P_{H^n\oplus H^n}\wt S_+(\l)=\begin{pmatrix} -m(\l)&
-M_{2+}(\l) \cr I_{H^n} & 0 \end{pmatrix}:H^n\oplus\cH_0'\to H^n\oplus H^n \label{3.24.1}\\
S_-(z)=P_{H^n\oplus H^n}\wt S_-(z)=\begin{pmatrix} -m(z)&
-M_{2-}(z) \cr I_{H^n} & 0 \end{pmatrix}:H^n\oplus\cH_1'\to H^n\oplus H^n   \label{3.24.2}
\end{align}
It follows from \eqref{1.37a} and \eqref{1.37b} that
\begin{equation}\label{3.26}
\wt Z_+(0,\l)=S_+(\l), \;\;\l\in\bC_+;\qquad \wt Z_-(0,z)=S_-(z),\;\; z\in\bC_-.
\end{equation}
Next assume that a collection $\pair\in\RH$ is given by \eqref{3.0}--\eqref{3.2} and let $Y_{\tau +}(\cd, \l),\; Y_{\tau -}(\cd, \l)$ be operator solutions of \eqref{1.23} with the initial data \eqref{3.10}, \eqref{3.11}. Moreover consider the operator functions
\begin{align*}
A_+(\l)=(-C_0(\l)\;\;\;C_1(\l)):\cH_0\oplus \cH_1\to \cK_+, \;\;\;\l\in\bC_+\;\;\\
A_-(\l)=(-\wt D_1(\l)\;\;\; \wt D_0(\l)):\cH_1\oplus \cH_0\to \cK_-, \;\;\;\l\in\bC_-,
\end{align*}
where $C_j(\l)$ and $\wt D_j(\l),\; j\in\{0,1\}$ are defined by \eqref{3.0}, \eqref{3.4} and \eqref{3.5}. Then
\begin{align*}
A_+(\l)\up H^n\oplus H^n =(-C_0(\l)\up H^n\;\;\;C_1(\l)\up H^n)=(-\hat C_2(\l)\;\;\hat C_1(\l)):H^n\oplus H^n\to \cK_+,\\
A_-(\l)\up H^n\oplus H^n =(-\wt D_1(\l)\up H^n\;\;\;\wt D_0(\l)\up H^n)=(-\hat D_2(\l)\;\;\hat D_1(\l)):H^n\oplus H^n\to \cK_-
\end{align*}
and the equalities \eqref{3.10} and \eqref{3.11} give
\begin{align*}
\wt Y_{\tau +}^*(0,\ov\l)=(\wt D_1(\l)-\wt D_0(\l)M_-(\l))^{-1}(-\hat D_2(\l)\;\;\;\hat D_1(\l))=\qquad\qquad\qquad\qquad\quad\\
\qquad\qquad\qquad\qquad\quad =(\wt D_1(\l)-\wt D_0(\l)M_-(\l))^{-1}A_-(\l)\up H^n\oplus H^n, \quad \l\in\bC_-\\
\wt Y_{\tau _-}^*(0,\ov\l)=( C_0(\l)- C_1(\l)M_+(\l))^{-1}(-\hat C_2(\l)\;\;\;\hat C_1(\l))=\qquad\qquad\qquad\qquad\quad\\
\qquad\qquad\qquad\qquad\quad =( C_0(\l)- C_1(\l)M_+(\l))^{-1}A_+(\l)\up H^n\oplus H^n, \quad \l\in\bC_+\qquad
\end{align*}
This and \eqref{3.26} yield
\begin{align}
\wt Z_+(0,\l) \wt Y_{\tau _-}^*(0,\ov\l)=P_{H^n\oplus H^n}\wt S_+(\l)( C_0(\l)- C_1(\l)M_+(\l))^{-1}A_+(\l)\up H^n\oplus H^n, \label{3.27}\\
\wt Z_-(0,\l) \wt Y_{\tau _+}^*(0,\ov\l)=P_{H^n\oplus H^n}\wt S_-(\l)(\wt D_1(\l)-\wt D_0(\l)M_-(\l))^{-1}A_-(\l)\up H^n\oplus H^n, \label{3.28}
\end{align}
where
\begin{align*}
\wt S_+(\l)( C_0(\l)- C_1(\l)M_+(\l))^{-1}A_+(\l)=\qquad\qquad\qquad\qquad\qquad\qquad\qquad
\qquad\\
\qquad\qquad\qquad=\begin{pmatrix} -M_+(\l)\cr I_{\cH_0} \end{pmatrix} ( C_0(\l)- C_1(\l)M_+(\l))^{-1}(-C_0(\l)\;\;C_1(\l))=\\
=\begin{pmatrix} M_+(\l)(C_0(\l)-C_1(\l)M_+(\l))^{-1} C_0(\l) & -M_+(\l) (C_0(\l)-C_1(\l)M_+(\l))^{-1} C_1(\l)\cr
-(C_0(\l)-C_1(\l)M_+(\l))^{-1} C_0(\l)&
(C_0(\l)-C_1(\l)M_+(\l))^{-1} C_1(\l)\end{pmatrix} =\\
= \wt\Om_{\tau +}(\l)-\tfrac 1 2 JD _{\cH_0,\cH_1}\qquad\qquad\qquad\qquad\qquad\qquad\qquad\qquad
\end{align*}
(here the last equality is implied by \eqref{3.20a}--\eqref{3.20d} and \eqref{1.0.0}). Hence by \eqref{3.22} and the obvious equality $J_{H^n}= P_{H^n\oplus H^n}J_{\cH_0,\cH_1}\up H^n\oplus H^n$ one obtains
\begin{align}\label{3.29}
\wt Z_+(0,\l) \wt Y_{\tau _-}^*(0,\ov\l)=P_{H^n\oplus H^n}( \wt\Om_{\tau +}(\l)-\tfrac 1 2 J_{\cH_0,\cH_1})\up H^n\oplus H^n=\Om_\tau (\l)-\tfrac 1 2 J_{H^n}
\end{align}
for all $ \l\in\bC_+$. Similarly starting with \eqref{3.28} and using the representation \eqref{3.21}--\eqref{3.21d} of $\wt\Om_{\tau -}(\l)$ one proves
\begin{align}\label{3.30}
\wt Z_-(0,\l) \wt Y_{\tau _+}^*(0,\ov\l)=\Om_\tau (\l)-\tfrac 1 2 J_{H^n}, \quad \l\in\bC_-.
\end{align}
Now \eqref{3.29} and \eqref{3.30} imply that the operator functions \eqref{3.12} obey
\begin{equation}\label{3.31}
\wt Z_0(0,\l) \wt Y_{\tau }^*(0,\ov\l)=\Om_\tau (\l)-\tfrac 1 2 J_{H^n}, \quad \l\in\bC_+\cup\bC_-.
\end{equation}
Next consider the corresponding Green function \eqref{3.13}. It follows from \eqref{1.25a} that $Z_0(x,\l)=Y_0(x,\l)\wt Z_0(0,\l)$ and $Y_\tau (x,\l)=Y_0(x,\l)\wt Y_\tau (0,\l)$. This and \eqref{3.31} give
\begin{align*}
G_\tau (x,t,\l)=Y_0(x,\l)(\wt Z_0(0,\l) \wt Y_\tau^* (0,\ov\l))Y_0^*(t,\ov\l)=\qquad\qquad\qquad\qquad\qquad\qquad\qquad \\
\qquad\qquad\qquad = Y_0(x,\l)(\Om_\tau (\l)-\tfrac 1 2 J_{H^n})Y_0^*(t,\ov\l), \;\;\text{if}\;\; x>t,\\
G_\tau (x,t,\l)=Y_0(x,\l)(\wt Y_\tau (0,\l)\wt Z_0^*(0,\ov\l) )Y_0^*(t,\ov\l)=\qquad\qquad\qquad\qquad \qquad\qquad\qquad\qquad\\
\qquad =Y_0(x,\l)(\Om^*_\tau (\ov\l)-\tfrac 1 2 J^*_{H^n})Y_0^*(t,\ov\l)= Y_0(x,\l)(\Om_\tau (\l)+\tfrac 1 2 J_{H^n})Y_0^*(t,\ov\l), \;\;\text{if}\;\; x<t,
\end{align*}
which is equivalent to \eqref{3.23}.
\end{proof}
Comparing \eqref{3.23} with \eqref{0.4} and taking Theorem \ref{th3.3} into account one concludes that $\Om_\tau (\cd)$ is a characteristic matrix in the sense of Shtraus \cite{Sht57,Bru74}. This allows us to introduce the following definition.
\begin{definition}\label{def3.5}
The operator function $\Om (\cd)=\Om_\tau(\cd):\bC_+\cup\bC_-\to [H^n\oplus H^n]$ defined by the relations \eqref{3.18}--\eqref{3.19} and \eqref{3.22} will be called the characteristic matrix of the operator $L_0$ corresponding to a collection $\pair\in\RH$.

Since by Theorem \ref{th3.1} formulas \eqref{3.8} and \eqref{3.9} give a bijective correspondence between all $\tau\in\RH$ and all generalized resolvents $\R$, the operator function $\Om (\cd)=\Om_\tau(\cd)$  will be also called the characteristic matrix of  $L_0$ corresponding to the generalized resolvents $\R=\bR_\tau (\l) $.

The characteristic matrix $\Om_\tau(\cd)$ will be called canonical if it corresponds to a canonical resolvent $\bR_\tau (\l)$ or, equivalently if $\tau\in \RZ$.
\end{definition}
\begin{remark}\label{rem3.5a}
To compare the representations \eqref{3.13} and \eqref{3.23} of the Green function note that in the case $\dim H=1$ (the scalar case) and $n_+(L_0)=n_-(L_0):=m$ the operator solution $Y_\tau (\cd,\l)$ in \eqref{3.13} contains less components than $Y_0(\cd,\l)$ in \eqref{3.23} (more precisely, $m$ and $2n$ components respectively). Therefore  we suppose that  formula  \eqref{3.13} for the Green function may be useful in the sequel.
\end{remark}
\subsection{Description of characteristic matrices and their properties}
It is clear that formulas \eqref{3.18}--\eqref{3.19} and \eqref{3.22} describe in fact all characteristic matrices of the operator $L_0$ by means of a "boundary" parameter $\tau\in\RH$. The same description in a rather different form is specified in the following theorem.
\begin{theorem}\label{th3.6}
Assume that $\Pi=\bta$ is a decomposing $D$-triplet \eqref{1.30} for $L$, $M_\pm(\cd)$ are the corresponding Weyl functions \eqref{1.32a}, \eqref{1.32b} and $S_\pm(\cd)$ are operator functions \eqref{3.24.1}, \eqref{3.24.2}. Moreover, let $A_0\in\exl$ be the maximal symmetric extension \eqref{1.33} and let $\bR_0(\l)$ be a generalized resolvent of $L_0$ given by
\begin{equation*}
\bR_0(\l)=\begin{cases}(A_0-\l)^{-1},\;\l\in\bC_+\cr  (A_0^*-\l)^{-1},\;\l\in\bC_-\end{cases}.
\end{equation*}
Then: (i) the characteristic matrix of the operator $L_0$, corresponding to the generalized resolvent  $\bR_0(\l)$ is
\begin{equation}\label{3.34}
\Om_0(\l)=\begin{pmatrix} m(\l) & -\tfrac 1 2 I_{H^n}\cr  -\tfrac 1 2 I_{H^n} & 0 \end{pmatrix}: H^n \oplus H^n \to H^n \oplus H^n, \quad \l\in\bC_+\cup\bC_-;
\end{equation}

(ii) the equality
\begin{equation}\label{3.35}
\Om (\l)(=\Om_\tau(\l))=\Om_0(\l)-S_+(\l)(\tau_+(\l)+M_+(\l))^{-1}S_-^*(\ov\l), \quad\l\in\bC_+
\end{equation}
establishes a bijective correspondence between all characteristic matrices $\Om(\l)$ of the operator $L_0$ and all collections $\pair\in\RH$. Moreover the characteristic matrix $\Om(\l)$ in \eqref{3.35} is canonical, if and only if $\pair\in\RZ$.
\end{theorem}
\begin{proof}
Let $\Om_0(\cd)$ be the operator function \eqref{3.34}. Then the equalities \eqref{3.22} and \eqref{3.18}--\eqref{3.18c} imply
\begin{align*}
\Om_\tau(\l)=P_{H^n\oplus H^n}\begin{pmatrix} M_+(\l) & -\tfrac 1 2 I_{\cH_1}\cr  -\tfrac 1 2 I_{\cH_0} & 0 \end{pmatrix}\up H^n\oplus H^n -\qquad\qquad\qquad\\
-P_{H^n\oplus H^n} \begin{pmatrix} -M_+(\l) \cr I_{\cH_0}\end{pmatrix}(\tau_+(\l)+M_+(\l))^{-1}(-M_-^*(\ov\l)\;\;\; I_{\cH_1})\up H^n\oplus H^n, \quad \l\in\bC_+,
\end{align*}
where
\begin{align*}
P_{H^n\oplus H^n}\begin{pmatrix} M_+(\l) & -\tfrac 1 2 I_{\cH_1}\cr  -\tfrac 1 2 I_{\cH_0} & 0 \end{pmatrix}\up H^n\oplus H^n=\begin{pmatrix} P_{H^n} M_+(\l)\up H^n & -\tfrac 1 2 I_{H^n} \cr -\tfrac 1 2 I_{H^n} & 0\end{pmatrix}= \Om_0(\l).
\end{align*}
Combining these relations with \eqref{3.23.1}--\eqref{3.24.2} we arrive at \eqref{3.35}.

Next assume that $\tau_0=\{\tau_{0+}, \tau_{0-}\}\in\RH$ is a collection given by
\begin{equation*}
\tau_{0+}(\l)\equiv \t_0:=\{0\}\oplus \cH_1\, (\in\CA),\;\; \l\in\bC_+; \quad \tau_{0-}(\l)\equiv \t_0^\tm,\;\;\l\in\bC_-.
\end{equation*}
Then in view of \eqref{1.11} the corresponding generalized resolvent $\bR_{\tau_0}(\l)$, defined by \eqref{3.8} and \eqref{3.9}, coincides with $\bR_0(\l)$. On the other hand, $(\tau_{0+}(\l)+M_+(\l))^{-1}=0,\; \l\in\bC_+$ and by \eqref{3.35} $\Om_{\tau_0}(\l)=\Om_0(\l)$. This implies the statement (i) of the theorem.
\end{proof}
The following lemma will be useful in the sequel.
\begin{lemma}\label{lem3.7}
Assume that $\gH,\cH$ and  $\cH_0$ are Hilbert spaces, $\cH_1$ is a subspace in $\cH_0$, $\cH_2:=\cH_0\ominus \cH_1$, $\Lambda$ is an open set in $\bC$ and $A(\cd):\Lambda\to [\cH\oplus\cH_0, \cH\oplus\cH_1], \; \wt\g(\cd):\Lambda\to [\cH\oplus\cH_0,\gH]$ are operator functions given by
\begin{align}
A(\l)=\begin{pmatrix} a_1(\l)& a_2(\l)\cr a_3(\l)& a_4(\l)\end{pmatrix}:\cH\oplus\cH_0\to \cH\oplus\cH_1, \quad \l\in\Lambda \label{3.36}\\
\wt\g(\l)=(\wt\g_1(\l)\;\; \wt\g_2(\l)):\cH\oplus\cH_0 \to\gH, \quad \l\in\Lambda\qquad\label{3.37}
\end{align}
and obeying the relation
\begin{equation}\label{3.38}
A(\mu)-A^*(\l)P_{\cH\oplus\cH_1}+i\,P_{\cH_2}=(\mu-\ov\l)\wt\g^*(\l)\wt\g (\mu), \quad \mu,\l\in\Lambda.
\end{equation}
Moreover let $\tau(\l)=\{K_0(\l),K_1(\l);\cK \}\,(\in\CA), \; \l\in\Lambda$ be a $\CA$-valued function (an operator pair, see \eqref{1.1a}) with $0\in\rho (\tau(\l)+a_4(\l))$ and let $\f_\tau(\cd):\Lambda\to [\cH]$ be a transform of $\tau(\cd)$ defined by
\begin{equation}\label{3.39}
\f_\tau(\l)=a_1(\l)-a_2(\l)(\tau(\l)+a_4(\l))^{-1}a_3(\l), \quad \l\in\Lambda.
\end{equation}
Then
\begin{align}
\f_\tau(\mu)-\f_\tau^*(\l)=(\mu-\ov\l)\g_\tau^*(\l)\g_\tau(\mu)+ \b^*(\l)(K_{01}^*(\l)K_1(\mu)- K_1^*(\l)K_{01}(\mu)-\label{3.40}\\
-i K_{02}^*(\l)K_{02}(\mu))\b (\mu), \quad\; \mu,\l\in \Lambda,\qquad\qquad\qquad\nonumber
\end{align}
where $\g_\tau(\l)(\in [\cH,\gH]),\; \b(\l)(\in [\cH,\cK])$ and $K_{0j}(\l) (\in [\cK, \cH_j]),\; j\in\{1,2\}, \; \l\in\Lambda$  are operator functions given by
\begin{align}
\g_\tau(\l)=\wt\g_1(\l)-\wt\g_2(\l)(\tau(\l)+a_4(\l))^{-1}a_3(\l),
\label{3.41}\\
 \b(\l)=(K_1(\l)+a_4(\l)K_0(\l))^{-1}a_3(\l),\quad\label{3.41a}\\
K_0(\l)=(K_{01}(\l)\;\;\, K_{02}(\l))^\top :\cK\to\cH_1\oplus\cH_2.\label{3.42}
\end{align}
\end{lemma}
\begin{proof}
First observe that \eqref{3.38} is equivalent to the relations
\begin{align}
a_1(\mu)-a_1^*(\l)=(\mu-\ov\l)\wt\g_1^*(\l)\wt\g_1(\mu), \;\;\;\; a_4(\mu)- a_4^*(\l)P_1+iP_2= (\mu-\ov\l)\wt\g_2^*(\l)\wt\g_2(\mu),\label{3.43}\\
a_2(\mu)- a_3^*(\l)P_1= (\mu-\ov\l)\wt\g_1^*(\l)\wt\g_2(\mu), \;\;\;\; a_3(\mu)- a_2^*(\l)= (\mu-\ov\l)\wt\g_2^*(\l)\wt\g_1(\mu), \label{3.44}
\end{align}
where $P_j$ is the orthoprojector in $\cH_0$ onto $\cH_j, \; j\in\{ 1,2\}$ and $\mu,\l\in\Lambda$. Letting
\begin{equation}\label{3.45}
\a_1(\l)=(K_1(\l)+a_4(\l)K_0(\l))^{-1}, \;\;\;\; \a(\l)=(\tau(\l)+a_4(\l))^{-1}=K_0(\l)\a_1(\l),
\end{equation}
one rewrites \eqref{3.39} and \eqref{3.41} as
\begin{equation}\label{3.46}
\f_\tau(\l)=a_1(\l)-a_2(\l)\a(\l)a_3(\l), \qquad \g_\tau(\l)= \wt\g_1(\l) -\wt\g_2(\l)\a(\l) a_3(\l).
\end{equation}
Combining these equalities with \eqref{3.43} and \eqref{3.44} one derives
\begin{align*}
(\mu-\ov\l)\g_\tau^*(\l)\g_\tau(\mu)=(\mu-\ov\l)(\wt\g_1^*(\l) -a_3^*(\l)\a^*(\l)\wt\g_2^*(\l))(\wt\g_1(\mu) -\wt\g_2(\mu)\a(\mu) a_3(\mu))=\\
a_1(\mu)-a_1^*(\l)-a_3^*(\l)\a^*(\l)(a_3(\mu)-a_2^*(\l))-(a_2(\mu)- a_3^*(\l)P_1)\a(\mu) a_3(\mu)+\quad\\
a_3^*(\l)\a^*(\l)(a_4(\mu)- a_4^*(\l)P_1+iP_2)\a(\mu) a_3(\mu).\qquad \qquad\qquad
\end{align*}
Thus
\begin{equation}\label{3.47}
(\mu-\ov\l)\g_\tau^*(\l)\g_\tau(\mu)=\f_\tau(\mu)-\f_\tau^*(\l)+ a_3^*(\l)X(\mu,\l)a_3(\mu),
\end{equation}
\begin{align*}
\text{where}\;\; X(\mu,\l)=\a^*(\l)\up\cH_1\cd a_4(\mu)\a(\mu)-\a^*(\l)a_4^*(\l)P_1\a(\mu)+i\a^*(\l)\up\cH_2\cd P_2\a(\mu)+\\
+P_1\a(\mu)-\a^*(\l)\up\cH_1.\qquad\qquad\qquad\qquad\qquad
\end{align*}
It follows from \eqref{3.45} that
\begin{align*}
P_1\a(\mu)=K_{01}(\mu)\a_1(\mu), \qquad \a^*(\l)\up\cH_1=\a_1^*(\l)K_{01}^*(\l), \\
\a^*(\l)\up\cH_2\cd P_2\a(\mu) =\a_1^*(\l)K_{02}^*(\l)K_{02}(\mu)\a_1(\mu)\qquad
\end{align*}
and, consequently,
\begin{align*}
X(\mu,\l)=\a_1^*(\l)K_{01}^*(\l)a_4(\mu)K_0(\mu)\a_1(\mu)-\a_1^*(\l) K_{0}^*(\l) a_4^*(\l) K_{01}(\mu)\a_1(\mu)+\qquad\\
i \a_1^*(\l)K_{02}^*(\l) K_{02}(\mu) \a_1(\mu) +K_{01}(\mu)\a_1(\mu)-\a_1^*(\l)K_{01}^*(\l) =\a_1^*(\l)[K_{01}^*(\l) a_4 (\mu)K_0(\mu)-\\
K_{0}^*(\l) a_4^*(\l) K_{01}(\mu)+iK_{02}^*(\l) K_{02}(\mu) +(K_{1}^*(\l)+K_0^*(\l)a_4^*(\l)) K_{01}(\mu)-K_{01}^*(\l)(K_1(\mu)+\\
a_4(\mu)K_0(\mu)) ] \a_1(\mu)= \a_1^*(\l)[K_{1}^*(\l) K_{01}(\mu)-K_{01}^*(\l)K_{1}(\mu)+i K_{02}^*(\l) K_{02}(\mu) ]\a_1(\mu).
\end{align*}
Substituting this equality in \eqref{3.47} we arrive at the required identity   \eqref{3.40}.
\end{proof}
\begin{proposition}\label{pr3.8}
Assume that $\Pi=\bta$ is a decomposing $D$-triplet \eqref{1.30} for $L$ and $M_+(\cd)$ is the corresponding Weyl function. Moreover let $\pair\in\RH$ be a collection of holomorphic operator pairs \eqref{3.0}--\eqref{3.2}. Then:

1) for every $\l\in\bC_+$ there exists the unique pair of operator functions $u_{j\tau}(\cd,\l)\in  \LH{H^n},$  $\; j\in\{1,2\}$ obeying the equation \eqref{1.23} and the boundary conditions
\begin{align}
(\hat C_1(\l)u_{j\tau}^{(1)}(0,\l)+\hat C_2(\l)u_{j\tau}^{(2)}(0,\l))\hat h+(C_0'(\l)\G_0'-C_1'(\l)\G_1')(u_{j\tau}(t,\l)\hat h)=\label{3.48}\\
=(-1)^{j-1}\hat C_j(\l)\hat h,\quad \hat h\in H^n, \;\; j\in\{1,2\}.\qquad\qquad\nonumber
\end{align}

2) The operator functions $u_{j\tau}(\cd,\l)$ are connected with the operator solution \eqref{1.38}  via
\begin{equation}\label{3.50}
u_{j\tau}(t,\l)=(-1)^{j-1}Z_+(t,\l)(C_0(\l)-C_1(\l)M_+(\l))^{-1}\hat C_j(\l), \;\;\,  j\in\{1,2\}, \;\; \l\in\bC_+.
\end{equation}
\end{proposition}
\begin{proof}
First observe that in view of \eqref{3.1} the conditions \eqref{3.48} are equivalent to
\begin{equation}\label{3.51}
(C_0(\l)\G_0-C_1(\l)\G_1)(u_{j\tau}(t,\l)\hat h)=(-1)^{j-1}\hat C_j(\l)\hat h,\quad \hat h\in H^n, \;\; j\in\{1,2\}.
\end{equation}
Let $u_{j\tau}(\cd,\l)\in \LH{H^n}, \; j\in\{1,2\}$ be operator solutions of \eqref{1.23} given by  \eqref{3.50}. Using Lemma \ref{lem1.2} one obtains
\begin{align*}
(C_0(\l)\G_0-C_1(\l)\G_1)(Z_+(t,\l)h_0)=(C_0(\l)\G_0-C_1(\l)\G_1) \g_+(\l)h_0=\\
=(C_0(\l)-C_1(\l)M_+(\l))h_0, \quad h_0\in\cH_0.\qquad
\end{align*}
Hence the operator functions $u_{j\tau}(\cd,\l)$ obey \eqref{3.51} and, consequently, \eqref{3.48}. Next assume that a pair of operator functions $u_j'(\cd,\l)\in\LH{H^n}$ also obeys \eqref{1.23} and   \eqref{3.48} and let $y_j:=u_{j\tau}(t,\l)\hat h-u_j'(t,\l)\hat h,\; \hat h\in H^n,\;\l\in\bC_+$. Consider the extension $\wt A(\l)\in\exl$ given by \eqref{3.8}. It is easy to see that $y_j\in \cD (\wt A(\l))\cap \gN_\l(L_0), \;\l\in\bC_+ $. At the same time $\l\in\rho (\wt A(\l))$ (see \eqref{3.9}) and, therefore, $\cD (\wt A(\l))\cap \gN_\l(L_0)=\{0\}$. Hence $y_j=0$ and, consequently, $u_{j\tau}(t,\l)=u_j'(t,\l), \; j\in\{1,2\}$. This implies the uniqueness of $u_{j\tau}(\cd,\l)$.
\end{proof}
Now we are ready to prove that the characteristic matrix $\Om_\tau(\cd)$  is a Nevanlinna operator function. This statement will be implied by an identity for $\Om_\tau(\cd)$, which, to our mind, is of self-contained interest.
\begin{proposition}\label{pr3.9}
Let the assumptions of Propositions \ref{pr3.8} be satisfied and let \eqref{1.6} be the equivalent to \eqref{3.0} representation of $\tau$. Moreover let $u_{j\tau}(\cd,\l), \; j\in\{1,2\}$ be operator solutions introduced in Proposition  \ref{pr3.8}, let
\begin{equation}\label{3.52}
U_\tau(t,\l):=(u_{2\tau}(t,\l)\;\; u_{1\tau}(t,\l)):H^n\oplus H^n\to H, \quad\l\in\bC_+
\end{equation}
and let $\b(\cd):\bC_+\to [H^n\oplus H^n,\cK_+']$ be the operator function given for every $\l\in\bC_+$ by the block-matrix representation
\begin{equation}\label{3.53}
\b(\l)=(-(K_1(\l)+M_+(\l)K_0(\l))^{-1}M_+(\l)\up H^n\;\;\,(K_1 (\l)+M_+(\l)K_0(\l))^{-1} \up H^n)
\end{equation}
Then: 1) the corresponding characteristic matrix $\Om_\tau(\cd)$  obeys the identity
\begin{align}
\Om_\tau(\mu)-\Om_\tau^*(\l)=(\mu-\ov\l)\int_0^b  U_\tau^*(t,\l) U_\tau(t,\mu)\, dt+ \b^*(\l) (K_{01}^*(\l)K_1(\mu)-\label{3.54}\\
-K_1^*(\l)K_{01}(\mu)-i K_{02}^*(\l)K_{02}(\mu))\b (\mu), \quad\; \mu,\l\in \bC_+,\qquad\nonumber
\end{align}
in which the  operator functions $K_{0j}(\cd):\bC_+\to [\cK_+',\cH_j], \;j\in\{1,2\} $ are defined by the block-matrix representation
\begin{equation*}
K_0(\l)=(K_{01}(\l)\;\;K_{02}(\l))^\top :\cK_+'\to \cH_1\oplus\cH_2, \quad \l\in\bC_+.
\end{equation*}
The integral in \eqref{3.54} converges strongly, that is
\begin{equation}\label{3.55}
\int_0^b  U_\tau^*(t,\l) U_\tau(t,\mu)\, dt = s-\lim_{\eta\uparrow b} \int_0^\eta  U_\tau^*(t,\l) U_\tau(t,\mu)\, dt.
\end{equation}

2) the characteristic matrix $\Om_\tau (\cd) $ satisfies the inequality
\begin{equation}\label{3.56}
\frac{Im\, \Om_\tau(\mu)} {Im \,\mu} \geq  \smallint\limits_0^b  U_\tau^*(t,\mu) U_\tau(t,\mu)\, dt, \quad \mu\in\bC_+,
\end{equation}
which turns into the equality for canonical characteristic matrices. Formula \eqref{3.56} implies that $\Om_\tau (\cd)$ is a Nevanlinna operator function.
\end{proposition}
\begin{proof}
1) It follows from \eqref{3.35} that $\Om_\tau(\cd)$ is a transform \eqref{3.39} of $\tau_+(\l)$ with operator coefficients $a_j(\l), \; j=1\div 4$, which are entries of the matrix
\begin{equation}\label{3.57}
A_\Pi(\l)=\begin{pmatrix} a_1(\l)& a_2(\l)\cr a_3(\l)& a_4(\l)\end{pmatrix}:=\begin{pmatrix} \Om_0(\l) & S_+(\l) \cr S_-^*(\ov\l) & M_+(\l)\end{pmatrix}(\in [(H^n\oplus H^n)\oplus\cH_0, (H^n\oplus H^n)\oplus\cH_1]).
\end{equation}
To prove \eqref{3.54} we apply Lemma \ref{lem3.7} to $A_\Pi(\l)$.

Assume that $\g_+(\cd)$ is the $\g$-field \eqref{1.13} and
\begin{equation}\label{3.58}
\wt \g_1(\l):=(-\g_+(\l)\up H^n\;\;\;0):H^n\oplus H^n\to\gH, \quad \wt\g_2(\l):=\g_+(\l)(\in [\cH_0,\gH])
\end{equation}
for all $\l\in\bC_+$. Let us show that the matrix \eqref{3.57} obeys the relations \eqref{3.43} and \eqref{3.44} (or, equivalently, \eqref{3.38}) with the operator functions \eqref{3.58}. To this end observe that the equalities \eqref{3.43} are immediate from \eqref{1.15.0}. Next, the first relation in \eqref{3.44} can be written in our case as
\begin{equation}\label{3.59}
S_+(\mu)-S_-(\ov\l)P_1=(\mu-\ov\l)\begin{pmatrix}-P_{H^n}\g_+^*(\l)\cr 0 \end{pmatrix}\g_+(\mu), \quad \mu,\l\in\bC_+,
\end{equation}
where $P_{H^n}$ is the orthoprojector in $\cH_0(=H^n\oplus\cH_0')$ onto $H^n$. To prove \eqref{3.59} note that $H^n\subset \cH_1$ and, consequently, $P_{H^n}P_1=P_{H^n}, \;P_{H^n}P_2=0 $. This and \eqref{3.24.1}, \eqref{3.24.2} yield
\begin{align*}
S_+(\mu)= (-P_{H^n}M_+(\mu)\;\; P_{H^n})^\top :\cH_0\to H^n\oplus H^n,\qquad\qquad\qquad\\
S_-(\ov\l)P_1=(-P_{H^n}M_+^*(\l)\;\;P_{H^n})^\top P_1= (-P_{H^n}M_+^*(\l)P_1 \;\; P_{H^n})^\top:\cH_0\to H^n\oplus H^n
\end{align*}
Moreover applying the operator $P_{H^n}$ to the identity \eqref{1.15.0} one derives
\begin{equation*}
P_{H^n}M_+(\mu)-P_{H^n}M_+^*(\l)P_1=(\mu-\ov\l)P_{H^n}\g_+^*(\l)\g_+ (\mu), \quad \mu,\l\in\bC_+.
\end{equation*}
Combining these equalities we arrive at \eqref{3.59}. Finally the second relation in \eqref{3.44} is equivalent to the first one.

Now by Lemma \ref{lem3.7} one has
\begin{align}
\Om_\tau(\mu)-\Om_\tau^*(\l)=(\mu-\ov\l)\g_\tau^*(\l)\g_\tau (\mu)+ \b^*(\l) (K_{01}^*(\l)K_1(\mu)-\label{3.60}\\
-K_1^*(\l)K_{01}(\mu)-i K_{02}^*(\l)K_{02}(\mu))\b (\mu), \quad\; \mu,\l\in \bC_+,\nonumber
\end{align}
where $\g_\tau(\cd)$ and $\b(\cd)$ are given by \eqref{3.41}  and \eqref{3.41a}. It follows from \eqref{3.24.2} that
\begin{equation}\label{3.61}
a_3(\l)=S_-^*(\ov\l)=(-M_+(\l)\up H^n\;\;\, I_{H^n}):H^n\oplus H^n\to\cH_1.
\end{equation}
Therefore the function $\b(\cd)$ in \eqref{3.60} is defined by \eqref{3.53} and to prove the identity \eqref{3.54} it remains to show that
\begin{equation}\label{3.62}
\g_\tau^*(\l)\g_\tau (\mu)=\smallint\limits_0^b  U_\tau^*(t,\l) U_\tau(t,\mu)\, dt:=s-\lim_{\eta\uparrow b} \smallint\limits_0^\eta U_\tau^*(t,\l) U_\tau(t,\mu)\, dt.
\end{equation}
Combining \eqref{3.61} with \eqref{3.58} and using next \eqref{3.17}  one obtains
\begin{align*}
\g_\tau(\l)=(-\g_+(\l)\up H^n\;\;\; \; 0)-\g_+(\l)(\tau_+(\l)+M_+(\l))^{-1}(-M_+(\l)\up H^n\;\;\; I_{H^n})=\\
=\g_+(\l)\bigl((-I_{\cH_0}+(\tau_+(\l)+M_+(\l))^{-1}M_+(\l))\up H^n\;\;\; -(\tau_+(\l)+M_+(\l))^{-1}\up H^n \bigr)=\\
=\g_+(\l)(C_0(\l)-C_1(\l)M_+(\l))^{-1}(-C_0(\l)\up H^n\;\;\;\,C_1(\l)\up H^n)=\quad\qquad\qquad\\
=\g_+(\l)(C_0(\l)-C_1(\l)M_+(\l))^{-1}(-\hat C_2(\l)\;\;\;\hat C_1(\l)).\qquad\qquad\qquad\qquad
\end{align*}
This and \eqref{1.40}, \eqref{3.50} give the equality
\begin{equation}\label{3.63}
(\g_\tau(\l)\wt h)(t)=U_\tau(t,\l)\wt h, \quad \wt h\in H^n\oplus H^n,\;\;\l\in\bC_+.
\end{equation}
Now applying Lemma \ref{lem1.8}, 3) to the operator solution $U_\tau (\cd,\l)$ one obtains
\begin{equation*}
\g_\tau^*(\l)f=\smallint\limits_0^b  U_\tau^*(t,\l) f(t)\, dt:=\lim_{\eta\uparrow b} \smallint\limits_0^\eta U_\tau^*(t,\l) f(t)\, dt, \quad f=f(t)\in\gH,
\end{equation*}
which together with \eqref{3.63} implies \eqref{3.62}.

The statement 2) of the proposition is immediate from the identity \eqref{3.54} and the inequality \eqref{1.7a}, which turns into the equality for $\tau\in\RZ$ (see Definitions \ref{def1.0.3} and \ref{def1.0.4}).
\end{proof}
\subsection{The case of equal deficiency indices }
In this subsection we suppose that the expression \eqref{1.20} obeys the condition $n_{b+}=n_{b-}$, in which case $n_+(L_0)=n_-(L_0)$. Moreover we assume that $\Pi=\bt$ is a decomposing boundary triplet \eqref{1.30} for $L$, so that $\cH_0'=\cH_1'=:\cH'$ and  $\cH_0=\cH_1=:\cH$ (see Definition \ref{def1.9b}). Our aim is to reformulate for this case the foregoing results on generalized resolvents and characteristic matrices.

First of all, Theorem \ref{th3.1} can be restated as follows.
\begin{theorem}\label{th3.10}
Under the above suppositions the relations
\begin{align}
\cD (\wt A(\l))=\{y\in\cD: \,\hat C_1(\l) y^{(1)}(0)+ \hat C_2(\l) y^{(2)}(0)+C'_0(\l)\G'_0 y-C'_1(\l)\G'_1 y=0\},\label{3.63a}\\
\R=\bR_\tau (\l)= (\wt A(\l)-\l)^{-1},\quad \l\in\bC_+\cup \bC_-\qquad\qquad\qquad\nonumber
\end{align}
(with $\wt A(\l)=L\up\cD (\wt A(\l))$) establish a bijective correspondence between all generalized resolvents $\R$ of the operator $L_0$ and all Nevanlinna operator pairs
\begin{equation}\label{3.64}
\tau (\l)=\{(C_0(\l),C_1(\l));\cK\}\in\wt R(\cH), \quad \l\in\bC_+\cup \bC_-
\end{equation}
(see Remark \ref{rem1.0.5}, 2)) defined by the block-matrix representations
\begin{align}
C_0(\l)=(\hat C_{2}(\l)\;\;C'_{0}(\l)): H^n\oplus\cH'\to\cK, \qquad \qquad\qquad \qquad\qquad \qquad\qquad  \label{3.65}\\
\qquad \qquad\qquad \qquad \qquad \qquad\qquad  C_1(\l)=(\hat C_{1}(\l)\;\;C'_1(\l)):H^n\oplus \cH'\to\cK , \quad \l\in\bC_+\cup\bC_-.\nonumber
\end{align}
Moreover $\R$ is a canonical resolvent if and only if $\tau\in\Rz$.
\end{theorem}
Next with every  operator pair $\tau(\cd)\in\Rh$ given by \eqref{3.64} and \eqref{3.65} we associate a family of operator solutions $Y_\tau(\cd,\l):\D\to [\cH,H] $ of the equation \eqref{1.23}, defined by the initial data
\begin{equation}\label{3.66}
\wt Y_{\tau }(0,\l)=(-\hat C_2^*(\ov\l)\;\;\; \hat C_1^*(\ov\l))^\top \, (C_0^*(\ov\l)-M(\l) C_1^*(\ov\l))^{-1}, \quad \l\in\bC_+\cup\bC_-
\end{equation}
(here $M(\l)$ is the corresponding Weyl function \eqref{1.42}). Then in view of \eqref{3.10}--\eqref{3.12} the Green function $G_\tau(x,t,\l)$,  corresponding to an operator pair $\tau\in\Rh$, is given by the equality \eqref{3.13}, in which $Y_\tau (\cd,\l)$ and $Z_0(\cd,\l)$ are  the operator solutions \eqref{3.66} and \eqref{1.43} respectively.

Now Theorem \ref{th3.3} implies the following result.
\begin{theorem}\label{th3.11}
Let the above assumptions be satisfied. Then the equalities \eqref{3.13} and \eqref{3.14} give a bijective correspondence between all generalized (canonical) resolvents $\R$ of the operator $L_0$ and all  operator pairs $\tau\in\Rh$ (respectively, $\tau\in\Rz$).
\end{theorem}
Assume that $\Pi=\bt$ is a decomposing boundary triplet \eqref{1.30} for $L$ and $M(\cd)$ is the corresponding Weyl function \eqref{1.42}. Then for every operator pair $\tau\in\Rh$ the corresponding operator functions \eqref{3.18} and \eqref{3.19} generate the operator function $\wt\Om_\tau(\cd):\bC_+\cup\bC_-\to [\cH\oplus\cH]$, defined by the block-matrix representation
\begin{equation}\label{3.67}
\wt\Om_{\tau }(\l)=\begin{pmatrix} M(\l)-M(\l)(\tau(\l) +M(\l))^{-1}M(\l) & -\tfrac 1 2 I_{\cH} +M(\l) (\tau(\l) +M(\l))^{-1} \cr -\tfrac 1 2 I_{\cH} +(\tau(\l) +M(\l))^{-1}M(\l) & -(\tau(\l) +M(\l))^{-1}\end{pmatrix}
\end{equation}
In view of \eqref{3.20a}--\eqref{3.21d} this function can be also represented as
\begin{multline}\label{3.68}
\wt\Om_{\tau }(\l)=\\
=\begin{pmatrix} M(\l)(C_0(\l)-C_1(\l)M(\l))^{-1} C_0(\l) & -\tfrac 1 2 I_{\cH} -M(\l) (C_0(\l)-C_1(\l)M(\l))^{-1} C_1(\l)\cr  \tfrac 1 2 I_{\cH}-(C_0(\l)-C_1(\l)M(\l))^{-1} C_0(\l) & (C_0(\l)-C_1(\l)M(\l))^{-1} C_1(\l)\end{pmatrix}
\end{multline}
where $C_0(\l)$ and $C_1(\l)$ are components of the operator pair $\tau $ (see \eqref{3.64}).  This and formula \eqref{3.22} show that the characteristic matrix $\Om_\tau(\cd)$ of the operator $L_0$, corresponding  to an operator pair $\tau(\cd)\in\Rh$, is defined by
\begin{equation}\label{3.69}
\Om_\tau (\l)= P_{H^n\oplus H^n}\,\wt \Om_{\tau }(\l)\up H^n\oplus H^n, \;\; \l\in\bC_+\cup\bC_-.
\end{equation}
The following theorem implied by Theorem \ref{th3.6} contains the description of all characteristic matrices of the operator $L_0$.
\begin{theorem}\label{th3.12}
Let  under the above assumptions $A_0\in\exl$ be the selfadjoint extension \eqref{1.33} and let $\wt S(\l)(\in [\cH,\cH\oplus\cH]), \; S(\l)(\in [\cH,H^n\oplus H^n])$ be operator functions given by
\begin{align*}
\wt S(\l):=(-M(\l)\;\;  I_{\cH})^\top :\cH
\to \cH \oplus \cH, \;\;\;\l\in\bC_+\cup\bC_-\qquad\qquad\qquad\qquad\\
S(\l)=P_{H^n\oplus H^n}\wt S(\l)=\begin{pmatrix} -m(\l)&
-M_{2}(\l) \cr I_{H^n} & 0 \end{pmatrix}:H^n\oplus\cH'\to H^n\oplus H^n,\;\;\;\l\in\bC_+\cup\bC_-.
\end{align*}
Then: (i) the characteristic matrix $\Om_0(\l)$ of the operator $L_0$ corresponding to the canonical resolvent $\bR_0(\l)=(A_0-\l)^{-1}$ is defined by \eqref{3.34};

(ii)the equality
\begin{equation}\label{3.70}
\Om(\l)(=\Om_\tau(\l))=\Om_0(\l)-S(\l)(\tau(\l)+M(\l))^{-1}S^*(\ov\l), \quad\l\in\bC_+\cup\bC_-
\end{equation}
establishes a bijective correspondence between all characteristic matrices $\Om(\l)$ of the operator $L_0$ and all Nevanlinna operator pairs $\tau(\cd)\in\Rh$. Moreover the characteristic matrix $\Om(\l)$ in \eqref{3.70} is canonical, if and only if $\tau(\cd)\in\Rz$.
\end{theorem}
Next combining Propositions \ref{pr3.8} and \ref{pr3.9} we arrive at the following result.
\begin{proposition}\label{pr3.13}
Let under the above suppositions $\tau (\cd)\in\Rh$ be a Nevanlinna operator pair \eqref{3.64}, \eqref{3.65} and let \eqref{1.9c} be an equivalent representation of $\tau(\cd)$. Then:

1) for every $\l\in\bC_+\cup\bC_-$ there exists the unique pair of operator functions $u_{j\tau}(\cd,\l)\in  \LH{H^n},$  $\; j\in\{1,2\}$ obeying the equation \eqref{1.23} and the boundary conditions \eqref{3.48}. Moreover,
\begin{equation*}
u_{j\tau}(t,\l)=(-1)^{j-1}Z_0(t,\l)(C_0(\l)-C_1(\l)M(\l))^{-1}\hat C_j(\l), \;\;\,  j\in\{1,2\}, \;\; \l\in\bC_+\cup\bC_-,
\end{equation*}
where $Z_0(\cd,\l)$ is the operator solution \eqref{1.43}.

2) the  identity \eqref{3.54} for the characteristic matrix $\Om_\tau(\cd)$ takes the form
\begin{align}
\Om_\tau(\mu)-\Om_\tau^*(\l)=(\mu-\ov\l)\int_0^b  U_\tau^*(t,\l) U_\tau(t,\mu)\, dt+ \b^*(\l) (K_{0}^*(\l)K_1(\mu)-\label{3.70a}\\
-K_1^*(\l)K_{0}(\mu))\b (\mu), \quad\; \mu,\l\in \bC_+\cup\bC_-,\qquad\qquad\nonumber
\end{align}
where $U_\tau(\cd,\l)$ is the operator solution \eqref{3.52} and $\b(\l)(\in [H^n\oplus H^n,\cK'])$ is given  by
\begin{equation*}
\b(\l)=(-(K_1(\l)+M(\l)K_0(\l))^{-1}M(\l)\up H^n\;\;\,(K_1 (\l)+M(\l)K_0(\l))^{-1} \up H^n).
\end{equation*}
Moreover, the inequality \eqref{3.56} holds for all $\mu\in \bC_+\cup\bC_- $.
\end{proposition}
\subsection{The case of minimal deficiency indices}
In this subsection we show that for differential expressions $l[y]$ with minimal deficiency indices $n_\pm (L_0)$ the previous  statements can be rather simplified.

Let
\begin{equation}\label{3.71}
K=(K_0\;\;\;K_1)^\top: H^n\to H^n\oplus H^n
\end{equation}
be a selfadjoint operator pair (that is $\t=\{K_0,K_1; H^n\}\in \C (H^n)$ is a selfadjoint linear relation). With each such a pair we associate the operator solution $\f_K(\cd,\l):\D\to [H^n,H] $ of the equation \eqref{1.23} defined by the initial data $\f_K^{(1)}(0,\l)= K_0, \; \f_K^{(2)} (0,\l)=K_1$.
\begin{proposition}\label{pr3.14}
Let $l[y]$ be a differential expression \eqref{1.20} with the deficiency indices $n_{b\pm}$ at the point $b$ (see Definition \ref{def1.6a}). Then the following statements are equivalent:

(i) $n_{b+}=n_{b-}=0$;

(ii) there exist a selfadjoint operator pair \eqref{3.71} and a pair of points $\l\in\bC_+$ and $\mu\in\bC_-$ such that
\begin{equation}\label{3.72}
\smallint_0^b ||\f_K(t,\l)\hat h||^2\,dt=\smallint_0^b ||\f_K(t,\mu) \hat h||^2\,dt=\infty, \quad \hat h\in H^n, \;\;\hat h\neq 0;
\end{equation}

(iii) the relation \eqref{3.72} holds for all selfadjoint pairs \eqref{3.71} and for all $\l\in \bC_+\cup\bC_-$.

If in addition $\dim H< \infty$, then the statements (i)--(iii) are equivalent to the equalities $n_+(L_0)=n_-(L_0)=n\cd\dim H$, which means minimality of the deficiency indices $n_\pm (L_0)$.
\end{proposition}
\begin{proof}
Let $\t=\{K_0,K_1; H^n\}\in \C (H^n)$ be a selfadjoint linear relation (operator pair) and let $L_\t\in\exl$ be a symmetric extension with the domain \eqref{1.29}. It is clear that for all $\l\in\bC_+ \cup\bC_-$ a defect subspace $\gN_\l (L_\t)$ consists of all functions $y\in\gH$ admitting the representation $y=\f_K(t,\l)\hat h$ with some $\hat h\in H^n$. At the same time according to Definition \ref{def1.6a} $n_{b\pm}=n_\pm(L_\t)=\dim \gN_\l (L_\t), \; \l\in\bC_\pm$. This and the relation
\begin{equation*}
y:=\f_K(t,\l)\hat h=0\iff K\hat h=0\iff \hat h=0
\end{equation*}
give the implications (i)$\Rightarrow$(iii) and (ii) $\Rightarrow$(i). Moreover, the implication (iii) $\Rightarrow$(ii) is obvious. Finally, the last statement is implied by the equality $n_\pm (L_0)=n\cd\dim H+n_{b\pm}$.
\end{proof}

\begin{proposition}\label{pr3.15}
Let the expression \eqref{1.20} satisfies at least one (and, hence, all) conditions (i)--(iii) from Proposition \ref{pr3.14}. Then:

1) a triplet $\Pi=\{H^n,\G_0,\G_1\}$ with
\begin{equation}\label{3.73}
\G_0 y=y^{(2)}(0), \quad \G_1 y=-y^{(1)}(0), \quad y\in\cD
\end{equation}
is a (unique decomposing) boundary triplet for $L$. Moreover,  the corresponding selfadjoint extension $A_0$ is given by
\begin{equation}\label{3.73a}
\cD(A_0)=\{y\in\cD: y^{(2)}(0)=0\}, \quad A_0=L\up\cD(A_0);
\end{equation}

2) for every $\l\in\rho (A_0)$ there exists the unique operator solution $v_0(\cd,\l)\in\LH{H^n}$ of the equation \eqref{1.23} obeying the condition $v_0^{(2)}(0,\l)=I_{H^n}$. Moreover, the $\g$-field $\g(\l)$ for the boundary triplet \eqref{3.73} is connected with $v_0(\cd,\l)$ by
\begin{equation*}
(\g(\l)\hat h)(t)=v_0(t,\l)\hat h, \quad \hat h\in H^n,\;\;\l\in\rho (A_0);
\end{equation*}

3) the $m$-function $m(\cd):\rho (A_0)\to [H^n]$ corresponding to the extension $A_0$ is uniquely defined by the condition
\begin{equation*}
-c(t,\l)m(\l)+ s(t,\l)\in\LH{H^n},\;\;\; \l\in\rho (A_0),
\end{equation*}
which imply the equality $m(\l)=-v_0^{(1)}(0,\l)$. Moreover, the Weyl function $M(\cd)$,corresponding to the boundary triplet \eqref{3.73}, is defined by
\begin{equation}\label{3.75}
M(\l)=m(\l), \quad \l\in \rho (A_0).
\end{equation}
\end{proposition}
\begin{proof}
The statement 1) is a consequence of the relation \eqref{1.32.0}, while all other statements are immediate from Corollary \ref{cor1.11}.
\end{proof}
Now applying Theorems \ref{th3.10} and \ref{th3.11} to the boundary triplet \eqref{3.73} we arrive at the following theorem.
\begin{theorem}\label{th3.16}
Let suppositions of Proposition \ref{pr3.15} be satisfied. Then:

1) the relations
\begin{align*}
\cD (\wt A(\l))=\{y\in\cD: \, C_1(\l) y^{(1)}(0)+  C_0(\l) y^{(2)}(0)=0\}, \quad  \wt A(\l)=L\up\cD (\wt A(\l)) \\
\R=\bR_\tau (\l)= (\wt A(\l)-\l)^{-1},\quad \l\in\bC_+\cup \bC_-\qquad\qquad\qquad
\end{align*}
establish a bijective correspondence between all generalized resolvents $\R$ of the operator $L_0$ and all Nevanlinna operator pairs
\begin{equation}\label{3.76}
\tau (\l)=\{(C_0(\l),C_1(\l));H^n\}\in\wt R(H^n), \quad \l\in\bC_+\cup \bC_-
\end{equation}
Moreover $\R$ is a canonical resolvent if and only if $\tau\in \wt R^0(H^n)$;

2) the Green function $G_\tau(x,t,\l)$, corresponding to the operator pair \eqref{3.76}, is
\begin{equation}\label{3.77}
G_\tau(x,t,\l)=\begin{cases}v_0(x,\l)\, \f_\tau^*(t,\ov\l), \;\; x>t \cr \f_{\tau}(x,\l)\, v_0^* (t,\ov\l), \;\; x<t \end{cases}, \quad \l\in\bC_+\cup\bC_-,
\end{equation}
where $v_0(\cd,\l)$ is defined in Proposition \ref{pr3.15}, 2) and $\f_\tau(\cd,\l):\D\to [H^n,H]$ is a solution of \eqref{1.23} with the initial data
\begin{equation*}
\wt \f_{\tau }(0,\l)=(- C_0^*(\ov\l)\;\;\;  C_1^*(\ov\l))^\top \, ( C_0^*(\ov\l)-m(\l) C_1^*(\ov\l))^{-1}, \quad \l\in\bC_+\cup\bC_-;
\end{equation*}

3) the statement of Theorem \ref{th3.11} holds with the equality \eqref{3.77} instead of \eqref{3.13}.
\end{theorem}
Next, Theorem \ref{th3.12} can be reformulated as follows.
\begin{theorem}\label{th3.17}
Let the expression \eqref{1.20} satisfies at least one (equivalently, all) conditions (i)--(iii) from Proposition \ref{pr3.14} and let $m(\cd)$ be the $m$-function corresponding to the extension \eqref{3.73a}. Then:

1) the characteristic matrix $\Om_\tau(\cd)$ of the operator $L_0$, corresponding to a Nevanlinna operator pair \eqref{3.76}, is defined by \eqref{3.67} (or, equivalently, \eqref{3.68}) with $\Om_\tau(\l)$ in place of $\wt\Om_\tau(\l)$ and $M(\l)=m(\l)$;

2) the statements of Theorem \ref{th3.12} hold  with $\tau(\cd)\in \wt R(H^n)$ and
\begin{equation*}
S(\l)=(-m(\l)\;\;I_{H^n})^\top:H^n\to H^n\oplus H^n, \quad M(\l)=m(\l),\;\;\;\l\in\bC_+\cup\bC_-.
\end{equation*}
\end{theorem}
Finally, Proposition \ref{pr3.13} gives the following result.
\begin{proposition}\label{pr3.18}
Let under the hypothesis of Theorem \ref{th3.17} $\tau(\cd)\in \wt R(H^n)$ be a Nevanlinna operator pair \eqref{3.76} and let $\tau(\l)=\{K_0(\l),K_1(\l);H^n\}$ be an equivalent representation of $\tau (\cd)$ (c.f. \eqref{1.9c}). Then:

1) for every $\l\in\bC_+\cup\bC_-$ there exists the unique pair of operator functions $u_{j\tau}(\cd,\l)\in  \LH{H^n},$  $\; j\in\{1,2\}$ obeying the equation \eqref{1.23} and the boundary conditions
\begin{equation*}
C_1(\l)u_{1\tau}^{(1)}(0,\l)+ C_0(\l)u_{1\tau}^{(2)}(0,\l))=C_1(\l), \quad C_1(\l)u_{2\tau}^{(1)}(0,\l)+ C_0(\l) u_{2\tau}^{(2)}(0,\l))=-C_0(\l).
\end{equation*}
Moreover, the  functions $u_{j\tau}(\cd,\l)$ are connected with the operator solution $v_0(\cd,\l)$ by
\begin{align*}
u_{1\tau}(t,\l)=v_0(t,\l)(C_0(\l)-C_1(\l)m(\l))^{-1} C_1(\l),\qquad\qquad\\
u_{2\tau}(t,\l)=-v_0(t,\l)(C_0(\l)-C_1(\l)m(\l))^{-1} C_0(\l), \;\; \l\in\bC_+\cup\bC_-.
\end{align*}

2) the characteristic matrix $\Om_\tau(\cd)$  obeys the identity \eqref{3.70a}, in which $U_\tau(\cd,\l)$  is the operator solution \eqref{3.52} and
\begin{equation*}
\b(\l)=(-(K_1(\l)+m(\l)K_0(\l))^{-1}m(\l)\;\;\;(K_1 (\l)+m(\l)K_0(\l))^{-1}):H^n\oplus H^n\to H^n.
\end{equation*}
Moreover, the inequality \eqref{3.56} holds for all $\mu \in\bC_+\cup\bC_-$.
\end{proposition}
\begin{remark}\label{rem3.19}
1) In the case $\dim H<\infty$ Proposition \ref{pr3.15} was proved in \cite{DM92}.

2) For a scalar Sturm-Liouville  differential expression $l[y]=-y''+q\,y $ on the semiaxis $[0,\infty)$  with deficiency indices  $n_\pm (L_0)=1$ (the limit point case) Theorem \ref{th3.16} and the statement 1) of Theorem \ref{th3.17} were proved in \cite{DLS87,DL96}. Observe also the papers \cite{DLS88,DLS93}, where similar results were obtained for  finite-dimensional Hamiltonian systems in the limit point case.
\end{remark}

\end{document}